\documentclass{article}

\usepackage{amsmath}
\usepackage{amssymb}
\usepackage{xcolor}
\usepackage{graphicx}
\usepackage{pgfplots}
\usepackage{caption}
\usepackage{subcaption}
\usepackage{authblk}

\usepackage{geometry}
\newcommand{\R}{\mathbb{R}}
\newcommand{\N}{\mathbb{N}}

\newcommand{\calF}{\mathcal{F}}
\newcommand{\calW}{\mathcal{W}}
\newcommand{\calP}{\mathcal{P}}

\newcommand{\abs}[1]{\left\lvert #1 \right\rvert}
\newcommand{\norm}[1]{\left\lVert #1 \right \rVert}

\newcommand{\set}[1]{\left\{ #1 \right\}}

\newcommand{\bbE}{\mathbb{E}}

\newtheorem{assumption}{Assumption}
\newtheorem{Remark}{Remark}
\newtheorem{Lemma}{Lemma}
\newtheorem{Theorem}{Theorem}
\newtheorem{proof}{proof}

\title{Feedback Loops in Opinion Dynamics of Agent-Based Models with Multiplicative Noise}

\author[a]{Nata\v sa Djurdjevac Conrad\footnote{natasa.conrad@zib.de}}
\author[a,c]{Jonas K\"oppl}
\author[b]{Ana Djurdjevac}

\affil[a]{Zuse Institute Berlin, Berlin, Germany}
\affil[b]{Freie Universit\" at Berlin, Institut f\"ur Mathematik und Informatik, Berlin, Germany}
\affil[c]{Weierstrass Institute for Applied Analysis and Stochastics Berlin, Germany}

\begin{document}
\date{}

\maketitle

\abstract{We introduce an agent-based model for co-evolving opinion and social dynamics, under the influence of multiplicative noise. In this model, every agent is characterized by a position in a social space and a continuous opinion state variable. Agents' movements are governed by positions and opinions of other agents and similarly, the opinion dynamics is influenced by agents' spatial proximity and their opinion similarity. Using numerical simulations and formal analysis, we study this feedback loop between opinion dynamics and mobility of agents in a social space.
We investigate the behavior of this ABM in different regimes and explore the influence of various factors on appearance of emerging phenomena such as group formation and opinion consensus. We study the empirical distribution and in the limit  of infinite  number of agents  we derive  a corresponding reduced model given by a partial differential equation (PDE). Finally, using numerical examples we show that a resulting PDE model is a good approximation of the original ABM.
}

\textbf{Keywords:} opinion dynamics, feedback loop,  agent-based modeling, multiplicative noise, group formation, mean-field limit, empirical distribution, stochastic partial differential equations.

\section{Introduction}
Opinion dynamics is one of the most important processes of our society, as our opinions influence not only our individual actions and behavior; but can also shape collective dynamics governing societal change and social movement. Complex interaction patterns between individuals and coupled social mechanisms in different environments are shown to be the crucial drivers of opinion dynamics \cite{lewandowsky2020technology}. With the introduction of online social media, the way people interact and share their opinions has drastically changed. For example, physical proximity is now not anymore a constrain for communication, everybody can engage in information transmission and express their opinions to a large number of people in different social, political and cultural environment. Additionally, large amounts of data became available about how people influence and are getting influenced in their opinions \cite{porten2020effects}, which provided new insights into the social mechanisms and emerging phenomena such as formation of echo chambers and opinion consensus.\\

During the last decades, extensive research has been carried out in order to 
understand how people shape their opinions in their social space, see recent reviews \cite{Kertesz2022, OpDynReview2016}. 
Governed by an increasing amount of available large-scale social data and fast computational advances, the topic of opinion dynamics gathered an interdisciplinary research community \cite{Kertesz2022,holley1975ergodic}. Existing work ranges from the studies on (1) \textit{model-driven approaches} that produce formal models for opinion dynamics that can be analysed using theories from mathematics and statistical physics to (2) \textit{data-driven approaches} that are used to explore empirical data using knowledge from social sciences. Using computer simulations and computational analysis, opinion dynamics models can be used as a tool for understanding social mechanisms, uncovering social interaction patterns and exploring influences of various factors on e.g. group formation and opinion consensus \cite{Schweitzer2000,starnini_emergence_2016, Kertesz2022,kan_adaptive_2021}. Furthermore, mathematical description of models is a starting point that enables the use of analytical tools. In such a way we can get theoretical predictions of the models, such as: long time behaviour, limiting behaviour of the system when the number of agents tends to infinity (macroscale) and description of the fluctuations in the case when the number of agents is very big but still finite (mesoscale). In addition, this type of analysis is a basis for developing rigorous numerical analysis which implies error estimates that should be expected  by numerical  computations and which determine the choice of parameters in the model that should be used in the experiments. However, most existing models are rather simple and rarely connect to empirical studies and available real-world data \cite{stauffer2009}. In order to close this gap between model- and data-driven approaches, new formal models need to be introduced that can better represent real-world social systems and capture complex mechanisms that govern how people shape their opinions.\\

The largest group of formal models for studying opinion dynamics are agent-based-models (ABMs), where a process of opinion formation takes place through interactions between individual agents. One example of such ABMs are \textit{Voter models} \cite{ClifSudb1973}, that describe opinion change between agents with discrete opinion states, where agents interaction dynamics is defined through an underlying social network. In the DeGroot model \cite{degroot_reaching_1974}, agents have continuous-valued opinions that are formed as the average opinions of all other agents. Further mathematical literature mostly focuses on \textit{bounded confidence models} \cite{rainer_hegselmann_opinion_2002}, where the dynamics of opinions depends only on interactions among agents that have similar opinions, without assuming an underlying interaction network, but rather all-to-all possible interactions. Different extensions of these models have been considered in order to account for more realistic scenarios. For example, introducing  stochastic effects \cite{schweitzer2003brownian,pineda2013noisy, goddard2022noisy, wang2017noisy, Schweitzer2000, gomes2020mean}, complex interaction mechanisms  \cite{CroAnt12,Schweitzer2021} and different types of agents (e.g. stubborn agents, influencers, campaigners) \cite{milli2021opinion, hegselmann2015opinion}. Additionally, in the context of complex social networks, the co-evolution of the opinion and the network dynamics has been studied, where the changes in the network structure influence the opinion dynamics and vice versa. It has been shown that this co-evolution process in network models governs the appearance of emerging structures, e.g. echo chambers \cite{Yu2017}. However, most of the existing ABMs for opinion dynamics do not include the dynamics of agents in a social space, despite that this process determines agents' interaction patterns. Extending on the models for epidemic spreading \cite{MobileAgentsEpidemic2014} and cultural dissemination \cite{Centola2007,Vazquez2007}, recently the so-called \textit{mobile agents} \cite{Starnini2020, starnini_emergence_2016} were introduced to account for the feedback between the spatial movement of agents and the social contagion dynamics. While simulation results on co-evolving dynamics (both for network models and ABMs) provided many useful insights into the social mechanisms behind opinion formation, theoretical considerations of such models are still largely missing.\\
In this article, we introduce a mathematical ABM, that includes feedback loops in opinion dynamics of agents moving in a social space, influencing and being influenced in their opinions. The focus of this manuscript is on the non-trivial two-way interaction between the agents' movement in a social space and their opinion dynamics. More precisely, agents' spatial movements can induce changes in opinion states over time and additionally, opinion dynamics can influence the spatial position of agents and opinion states of agents in their vicinity. This feedback loop between spatial and opinion changes is at the core of the systems co-evolving dynamics. We consider that both spatial and opinion dynamics are governed by stochastic dynamics with multiplicative noise, which generalizes the case of constant, additive noise that is usually considered \cite{pineda2013noisy, goddard2022noisy, wang2017noisy, Schweitzer2000, gomes2020mean}. We explore the impact the feedback has on the behavior of the system and in particular on the grouping of agents in opinion and/or social space. 
From the mathematical point of view, our ABM can be seen as an interacting particle system. In particular, on the microscopic level (level of agents) we formulate our model as a system of coupled stochastic differential equations (SDEs) with multiplicative noise. From the application point of view, the interest is to have weak regularity assumptions on the drift and diffusion coefficients. Next, we study the corresponding limiting equation in the case when the number of agents tends to infinity, i.e. the so called McKean-Vlasov equation. Furthermore, we show the well-posedness results for the limiting system, which is a very popular and challenging topic in the field of stochastic analysis. Nowadays, there is extensive literature about these results,  the standard results are \cite{sznitman_topics_1991, gartner1988mckean, krylov2005strong} and see the references therein. Another interesting class of results that is discussed in the literature and that we also consider, is the so-called propagation of chaos, meaning that one wants to prove the convergence of the microscopic model to McKean-Vlasov SDEs. The challenge in these proofs lies in the weak assumptions on the coefficients. Since the topic of this article is not the theoretical investigation of the weak regularity assumption, in order to illustrate our message we concentrate on the simple case of Lipschitz bounded coefficients with multiplicative noise case.  Furthermore, due to the high computational cost of ABM simulations when the number of agents is large, we suggest the standard model reduction approach that considers instead the empirical density rather than each agent individually. We derive the formal equations of the empirical density and its so called hydrodynamic limit. These results are in the spirit of the standard, so called Dean-Kawasaki equation \cite{dean_langevin_1996}. In the setting of the social dynamics, this model reduction for the uncoupled system has already been considered in \cite{HelfmannDjurdjevacConradDjurdjevacetal.2021}. Using a numerical example, we illustrate the expected behaviour of the system on the macroscopic scale, that is given by the partial differential equation.

The article is organized as follows. Our agent-based model for opinion dynamics with feedback loops is introduced and studied through numerical simulations in Section \ref{sec:Model}. Next, we develop a theoretical framework for studying the system at the macroscopic level by an mean-field approach in Section \ref{sec:MeanField} and we present the well-posedness result of the McKean-Vlasov SDE system with Lipschitz coefficients and the convergence results of propagation of chaos. In Section \ref{sec:EmpiricalMeasure} we present formal derivation of the equation that describes the dynamics of the empirical measure and its hydrodyanamic limit. We illustrate the limiting behaviour of the system on the macroscopic scale, by a numerical example. Finally, we derive our conclusions and possible future directions in Section \ref{sec:Discussion}.

\section{Model Description}\label{sec:Model}
We consider a closed system of $N$ interacting agents and agents' co-evolving opinion and social dynamics. At time $t\in[0,T]$, every agent $k,\ k=1,\dots,N$ has a position state $X^k_t \in \R^d$ and an opinion state $\Theta^k_t\in \R$. The position state $X^k_t$ of an agent is a point in an abstract social space, such that a distance between two agents refers to their social proximity that is described by their social similarity. In real-world social systems, information about a position in social space may be inferred from e.g. online social media. The opinion $\Theta^k_t$ of an agent $k$ is considered to be a continuous variable. For more generality, this model can be extended to incorporate several opinion entities, such that $\Theta^k_t\in\R^m$. However, for technical simplicity in this paper we will assume that $m=1$.
The state of the system at time $t \geq 0$ for the set of $N$ agents is given by
\begin{equation*}
Z_t = (X_t, \Theta_t) \in (\R^d)^N \times \R^N,
\end{equation*}
where the $k$-th row of the systems' state corresponds to the state $Z^k_t = (X^k_t, \Theta^k_t)$ of the $k-$th agent. All agents follow the same rules that describe how their positions and opinions change. More precisely, agents move in a social space governed by the position of other agents and their opinions. Similarly, the opinion states of agents are influenced by both agents' spatial proximity and their opinions. This feedback loop between spatial and opinion changes determines the systems adaptive dynamics. Additionally, to be able to account for external influences on agents and the sometimes seemingly random nature of human interactions, we model this system via a coupled system of stochastic differential equations with multiplicative noise of the form
\begin{equation}\label{eq:diffusion_process_all}
\begin{cases}
    dX_t&= \tilde{U}(X_t,\Theta_t) dt  +\sigma_{sp}(X_t,\Theta_t) dB^{sp}_t, \\\
    d\Theta_t &= \tilde{V}(X_t, \Theta_t) dt + \sigma_{op}(X_t,\Theta_t) dB^{op}_t,  
\end{cases}
\end{equation}
where:
\begin{itemize}
    \item $\tilde{U}: (R^d)^N \times \R^N \to (\R^d)^N$ is a \textit{spatial interaction map} that models how the positions and opinions of the agents influence the spatial movement of the agents, 
    \item $\tilde{V}: (R^d)^N \times \R^N \to (\R^d)^N$ is an \textit{opinion interaction map} that models how the positions and opinions of the agents influence the opinion states of the agents, 
    \item $B^{sp}$ and $B^{op}$ are independent Brownian motions starting in $0$, 
    \item $\sigma_{sp}(X_t,\Theta_t), \sigma_{op}(X_t,\Theta_t)$ are diffusion coefficients for spatial and opinion dynamics respectively. 
\end{itemize}
Dynamics given by the SDEs in \eqref{eq:diffusion_process_all} is rather abstract and in this form it doesn't provide much intuition on how it can be adapted to known social mechanisms coming from real-world systems. Thus, in the following we will focus on how complex interaction patterns and stochastic influences can be enforced in our model.

\subsection{Pairwise interactions}
We start by exploring the simplest type of interaction dynamics, namely pairwise (or 2-body) interactions that only take into account interactions between pairs of agents. More precisely, we consider the case where the interaction maps $\tilde{U}, \tilde{V}$ are linear functions of simpler interaction maps $U, V$, i.e. we  define $\tilde{U}= (\tilde{U}_1,...,\tilde{U}_N): (R^d)^N \times \R^N \to (\R^d)^N$ with 
\begin{align}
    \tilde{U}_k(X, \Theta) = \frac{1}{N}\sum_{j=1}^N U(X^k, X^j, \Theta^k, \Theta^j), 
\end{align}
for some pair-interaction map $U: \R^d \times \R^d \times \R \times \R \to \R^d$ and analogously for $\tilde{V}$. In this model, agents shape their opinions based on the mean opinion of other agents through pairwise interactions, which is the setting of many of the classical models for opinion dynamics  \cite{degroot_reaching_1974,beckmann_interacting_2003,rainer_hegselmann_opinion_2002}. 
Thus, the dynamics of the $k$-th agent are given by
\begin{align}\label{eq:genSystem}
    \begin{cases}
    dX^k_t=  \frac{1}{N}\sum_{j=1}^N U(X^k_t,X^j_t, \Theta^k_t, \Theta^j_t) dt  +\sigma_{sp}(X_t,\Theta_t) dB^{sp,k}_t, \\\
    d\Theta^k_t = \frac{1}{N}\sum_{j=1}^N V(X^k_t,X^j_t, \Theta^k_t,  \Theta^j_t) dt + \sigma_{op}(X_t,\Theta_t) dB^{op,k}_t.  
    \end{cases}
\end{align}
Note that in order to simplify the notation, we do not write $X_t^{k,N}$, $\Theta_t^{k,N}$, but instead $X^k_t$,  $\Theta_t^k$ respectively. As a more concrete example one can consider the following model that can be seen as an extension of a classical model by DeGroot \cite{degroot_reaching_1974}, in a sense that it does not assume interactions between all agents, but instead assumes that two agents can only interact if their positions in a social space are closer than a certain interaction radius $R$, like in the bounded confidence models. The reason for such modeling decision is that agents that are further away in a social space, i.e. are having a low social similarity, may have conflicting attitudes and social norms, and thus lack motivation to interact with each other. An opinion interaction map that models this idea is 
\begin{align}\label{eq:OpModel}
    V(x_1, x_2, \theta_1, \theta_2) := \alpha\cdot 1_{[0,R]}\left(\norm{x_1 -x_2}\right)\cdot(\theta_2 - \theta_1),
\end{align}
where $\alpha$ is the \textit{opinion strength} parameter and $\norm{\cdot}$ refers to the Euclidean distance. In our model, $\alpha$ regulates the strength of the social influence on the agent's opinion, i.e. the higher the $\alpha$ the more influence the pairwise opinion difference has on an agent's opinion. As given by \eqref{eq:OpModel}, this model formalizes the idea that two agents can only interact if their positions in a social space are close enough and if they interact, then their opinions become more similar. Since we do not distinguish between individual agent types, we assume that $\alpha$ is a constant, i.e. it is equal for all agents in the system. In the literature there are different variations of this classical interaction dynamics \cite{friedkin_social_1999,Schweitzer2021}. As discussed above, our model extends these by introducing the movements of agents in a social space and its feedback loop with the opinion dynamics. An example of such dynamics is given by the spatial interaction map 
\begin{align}\label{eq:SpModel} 
    U(x_1, x_2, \theta_1, \theta_2) := \beta\cdot 1_{[0,R]}\left(\norm{x_1 - x_2}\right) \cdot \text{sgn}(\theta_1 \cdot \theta_2)\cdot (x_2 - x_1),
\end{align}
where $\beta$ denotes the \textit{spatial strength} parameter regulating how much agents compromise when updating their positions. The third term in \eqref{eq:SpModel} introduces a direction of agents' spatial motion based on their opinions, i.e. agents can either attract or repel each other, depending on whether their opinions are similar or different to each other. Thus, movement of an agent in a social space is determined by its spatial closeness and opinion similarity with other agents in the system.

\subsection{Multi-body interactions}

Most existing models of opinion dynamics only take pairwise interactions into account, ignoring any higher order interaction dynamics. While this simplifies mathematical considerations significantly, this is a very rough approximation of real life interactions and does not allow to model group effects such as peer pressure, see e.g. \cite{battiston_networks_2020, battiston_physics_2021}.
Recently there has been an increasing interest in going beyond pairwise interactions to create more realistic models for social dynamics, e.g. \cite{neuhauser_opinion_2021}. Our model can be extended with such effects by including the spatial interaction dynamics as introduced in the previous section by \eqref{eq:SpModel} and considering multi-body interactions for the opinion dynamics proposed in \cite{neuhauser_opinion_2021}. In particular, the effect of peer pressure within groups of agents in spatial proximity can be modeled by the opinion interaction map 
\begin{align}
    V(x_1, x_2, x_3, \theta_1, \theta_2, \theta_3) := \alpha \cdot \left(\prod_{i,j=1}^3 1_{[0,R]}(\norm{x_i - x_j}) \right) s(\abs{\theta_2 - \theta_3}) [(x_2 - x_1) + (x_3 - x_1)],
\end{align}
where $s: [0,\infty) \to \R$ is a non-increasing positive-definite function, e.g. $s(x) = \exp(\lambda x)$, for some $\lambda < 0$. This means, that the influence of agents $2$ and $3$ on the opinion of agent $1$ is stronger if they have similar opinions. Different from the classical DeGroot model, in a case when multi-body interactions are defined on hypergraphs, shifts in the average
opinion of the system have been observed \cite{neuhauser_opinion_2021}. Similar results are to be expected in our model and will be the topic of our future work. Here, in order to keep our analytical results more trackable, we will focus on the case of pairwise interactions.  

\subsection{Stochastic influence: Multiplicative noise} \label{sec:MultNoise}
Many of the classical models of opinion dynamics are of a deterministic nature  \cite{degroot_reaching_1974,rainer_hegselmann_opinion_2002,beckmann_interacting_2003}. Historically, the first stochastic versions of these models have always used additive noise, i.e. stochastic noise with constant (in time, status and over the population) strength \cite{wang_noisy_2017,pineda_noisy_2009}. Having a constant noise coefficient also implies that the driving noise of agents $i$ and $j$ are independent for all $i,j \in \set{1,\dots,N}$. This makes the model mathematically easier, but such a modelling choice is questionable from a social sciences perspective. 
The usual role of the noise is to account for external influences, that are not already incorporated into the model through agents interactions, such as randomly occurring environmental changes or some other significant events. However, these events influence all (or at least most) agents, introducing non-trivial correlation between the individual noises. The choice of additive noise also doesn't account for more complex social mechanisms observed in real world systems \cite{starnini_emergence_2016}. Namely, the strength of the noise can additionally depend on how homogeneous the opinions of agent's peers are. In a highly polarised group of individuals, one might be more susceptible to random influences than in a group where everyone has roughly the same opinion.\\
Recently, a few models considered such complex mechanisms by introducing a \textit{multiplicative noise}, e.g. in the context of opinion dynamics \cite{starnini_emergence_2016,Maas2010}, animal movement \cite{preisler2004modeling} and flocking of a Cucker-Smale system \cite{sun2015positive}. 

Extending on the ideas from \cite{starnini_emergence_2016}, one concrete example of a multiplicative noise for our model is defined by 
\begin{align}\label{eq:MultNoise}
    \sigma_{op}^i(X_t, \Theta_t) = \min_{j: \norm{X_t^i - X_t^j} \leq R}\abs{\Theta_t^i - \Theta_t^j} \cdot \text{Id}, i = 1,\ldots,N. 
\end{align}

Thus, the noise in this model is characterized by a similarity bias, i.e. the closer the opinions of agents' peers are to the opinion of that agent, the smaller the random fluctuations. Similarly, agents have a higher probability to move away from agents with very different opinions. These effects have been shown to yield fast formation of stable spatial clusters in which a local consensus is reached  \cite{starnini_emergence_2016}.

\subsection{Numerical simulations of the ABM}\label{sec:Simulation}
In this section, we show the main properties of our proposed model for opinion dynamics with feedback loops given by \eqref{eq:genSystem}, \eqref{eq:OpModel} and \eqref{eq:SpModel}. Using numerical simulations, we study how different parameters influence grouping of agents into social and opinion clusters. Additionally, we demonstrate a difference between additive and multiplicative noise and how this affects the stability of clusters and opinion distribution within clusters.\\

\begin{figure}[htb!]
	\centering
	\begin{subfigure}[b]{0.3\textwidth}
    \centering
	\includegraphics[width=\textwidth]{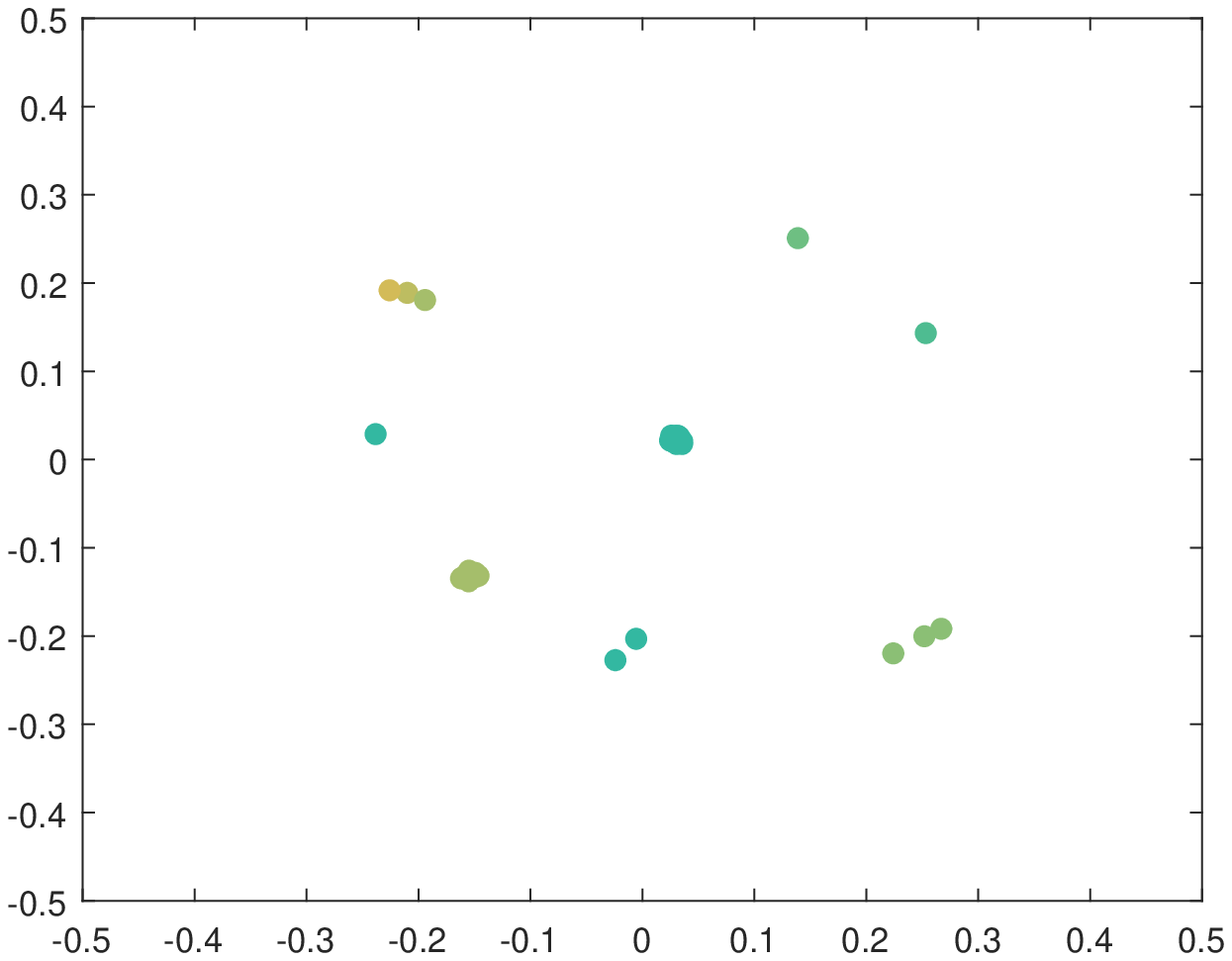}
	  \caption{$\sigma = 0.01$}
      \label{subfig:AddNoise01}
     \end{subfigure}     
     \hfill
     \begin{subfigure}[b]{0.3\textwidth}
    \centering
	\includegraphics[width=\textwidth]{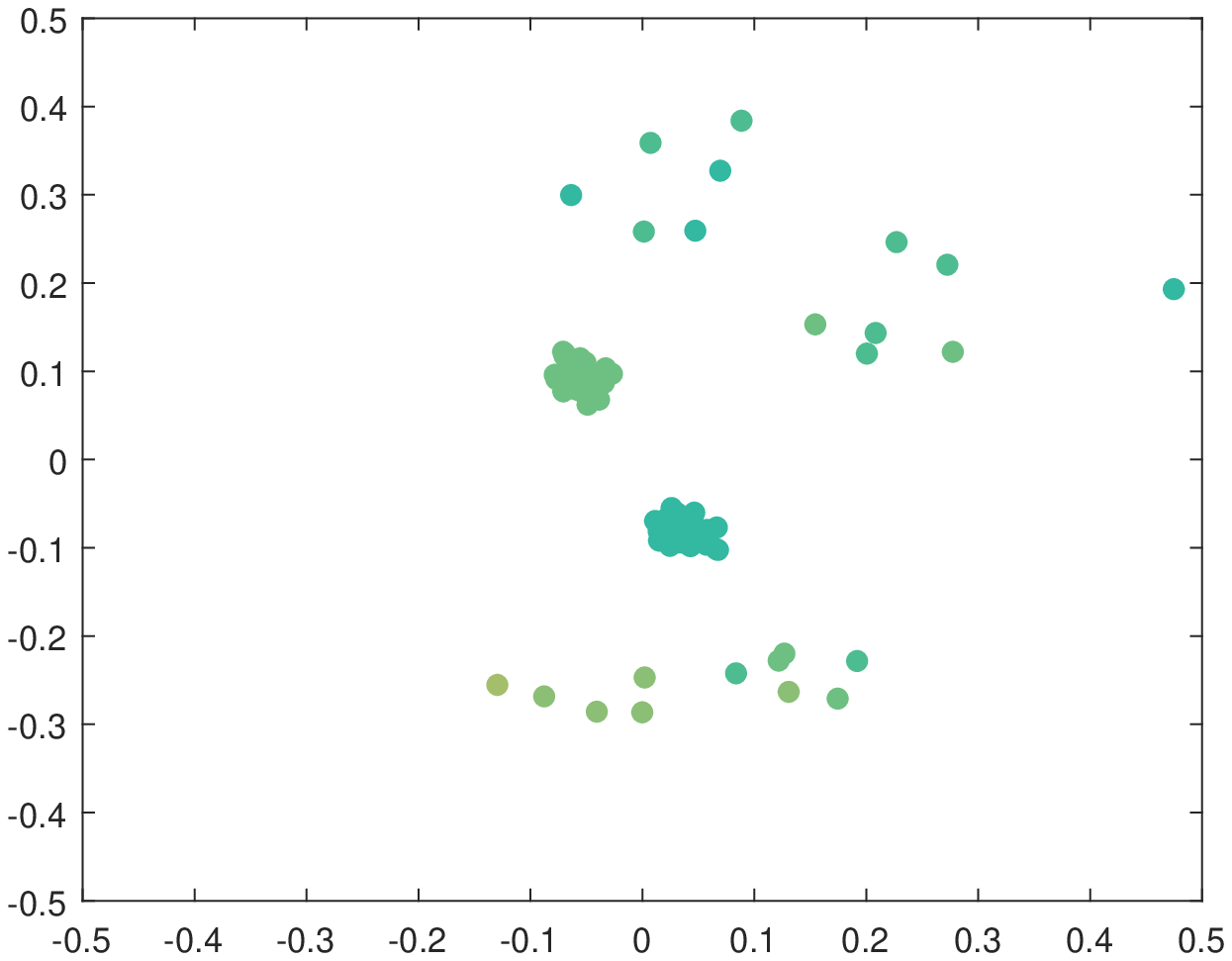}
		         \caption{$\sigma = 0.05$}
        \label{subfig:AddNoise05}
     \end{subfigure}          
     \hfill
	\begin{subfigure}[b]{0.3\textwidth}
    \centering
	\includegraphics[width=\textwidth]{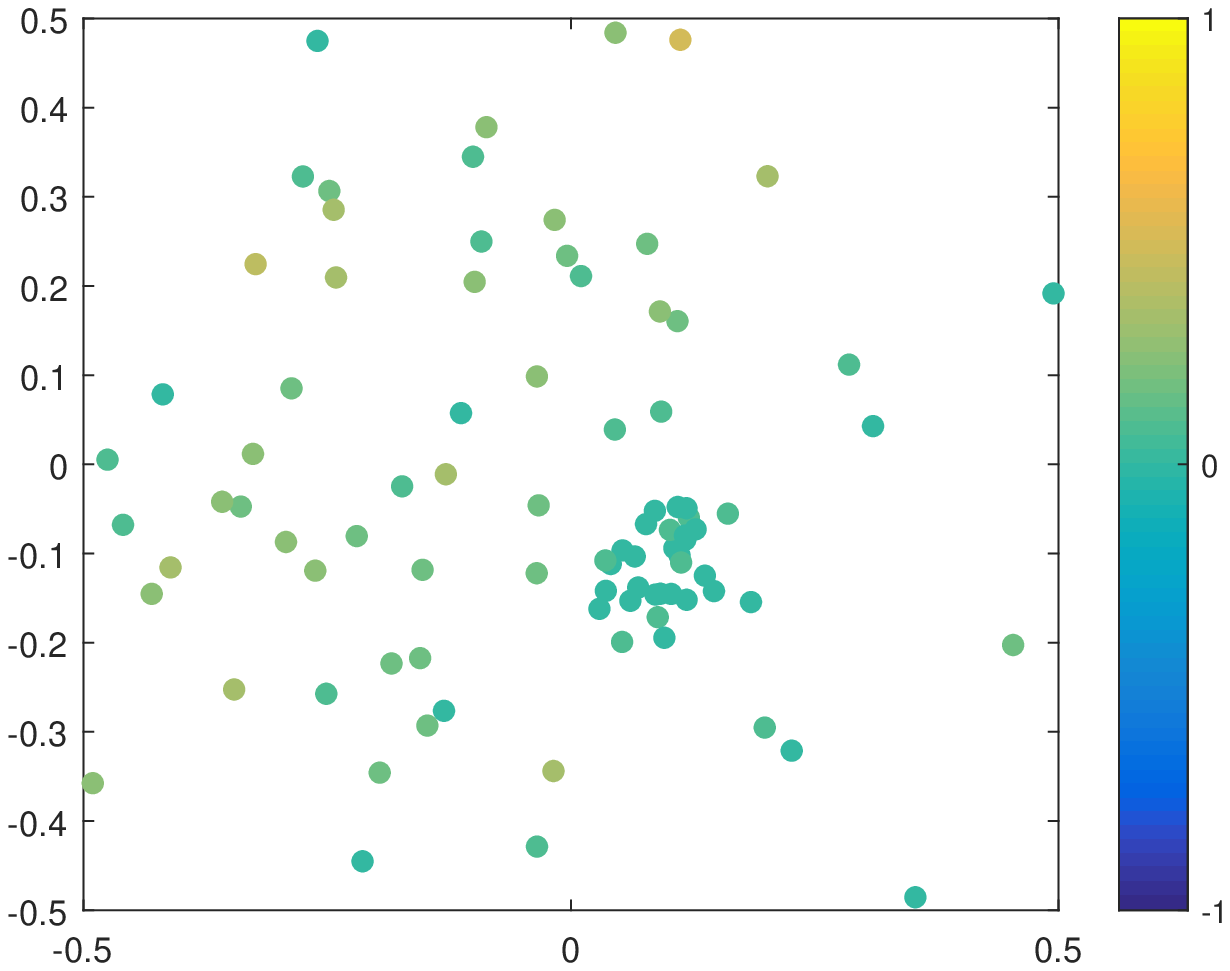}		         \caption{$\sigma = 0.15$} \label{subfig:AddNoise1}
	\end{subfigure}\\
	\caption{Snapshots from numerical simulations at final time $T=2.5$ for different influences of additive noise: (a) $\sigma = 0.01$, (b) $\sigma = 0.05$ and (c) $\sigma = 0.15$. Position of agents indicated their position in a social space. Color of agents denotes their opinions according to the color-bar. Other parameters are fixed to $R=0.015$ and $\alpha = \beta = 20$.}
	\label{fig:additive_example}
\end{figure}

\begin{figure}[htb!]
	\centering
	\begin{subfigure}[b]{0.3\textwidth}
    \centering
	\includegraphics[width=\textwidth]{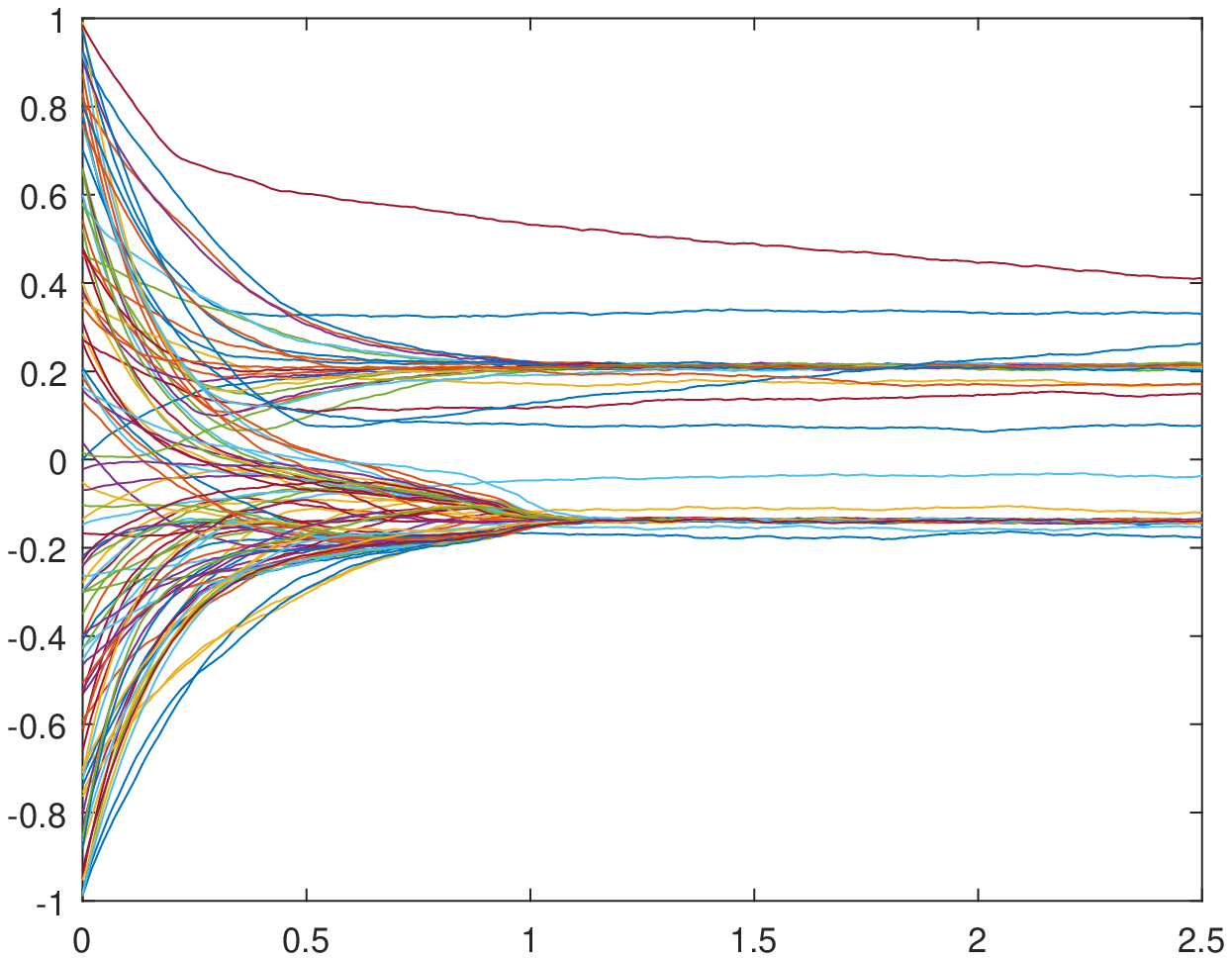}
	  \caption{$\sigma = 0.01$}
      \label{subfig:TrAddNoise01}
     \end{subfigure}     
     \hfill
     \begin{subfigure}[b]{0.3\textwidth}
    \centering
	\includegraphics[width=\textwidth]{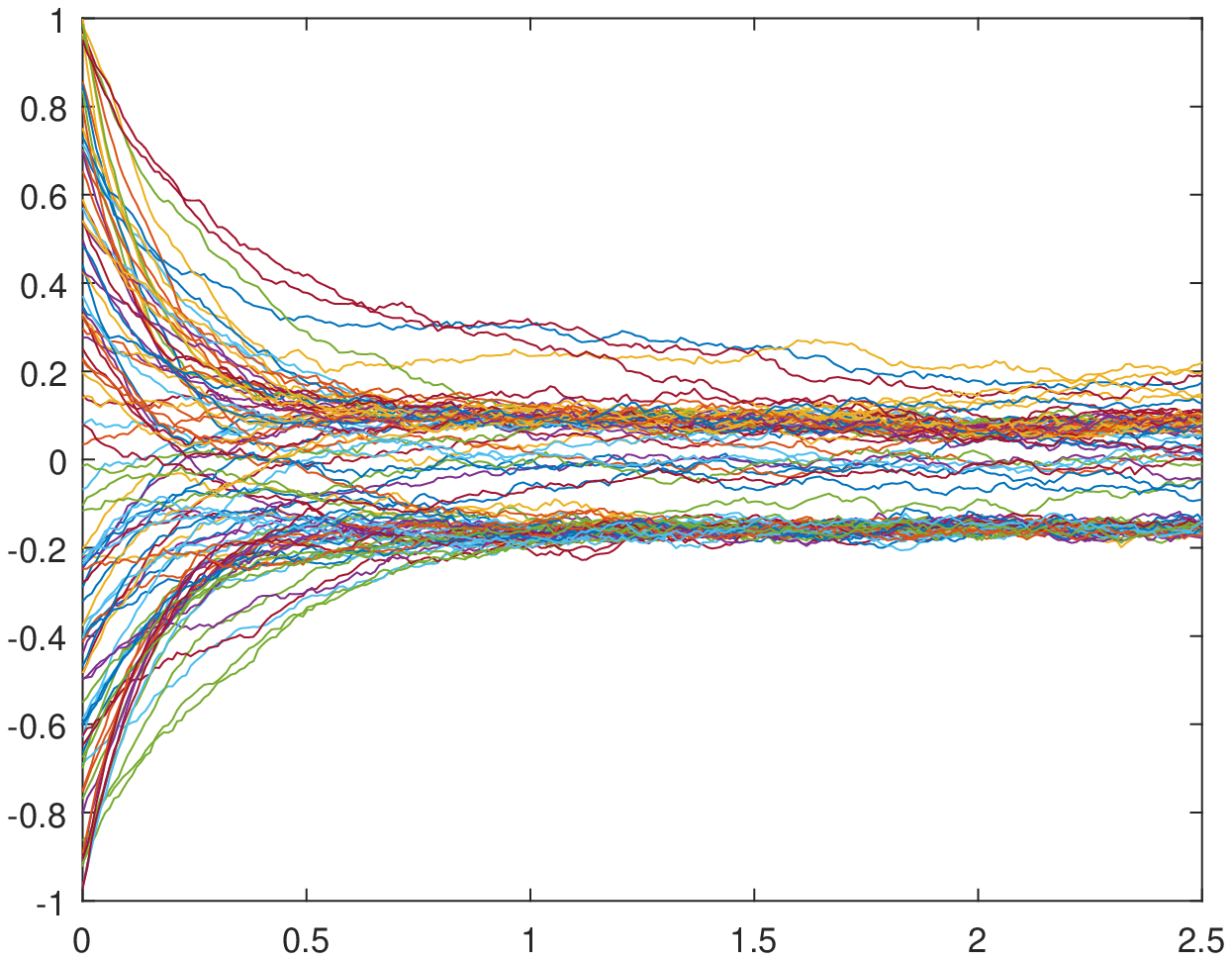}
		         \caption{$\sigma = 0.05$}
        \label{subfig:TrAddNoise05}
     \end{subfigure}          
     \hfill
	\begin{subfigure}[b]{0.3\textwidth}
    \centering
	\includegraphics[width=\textwidth]{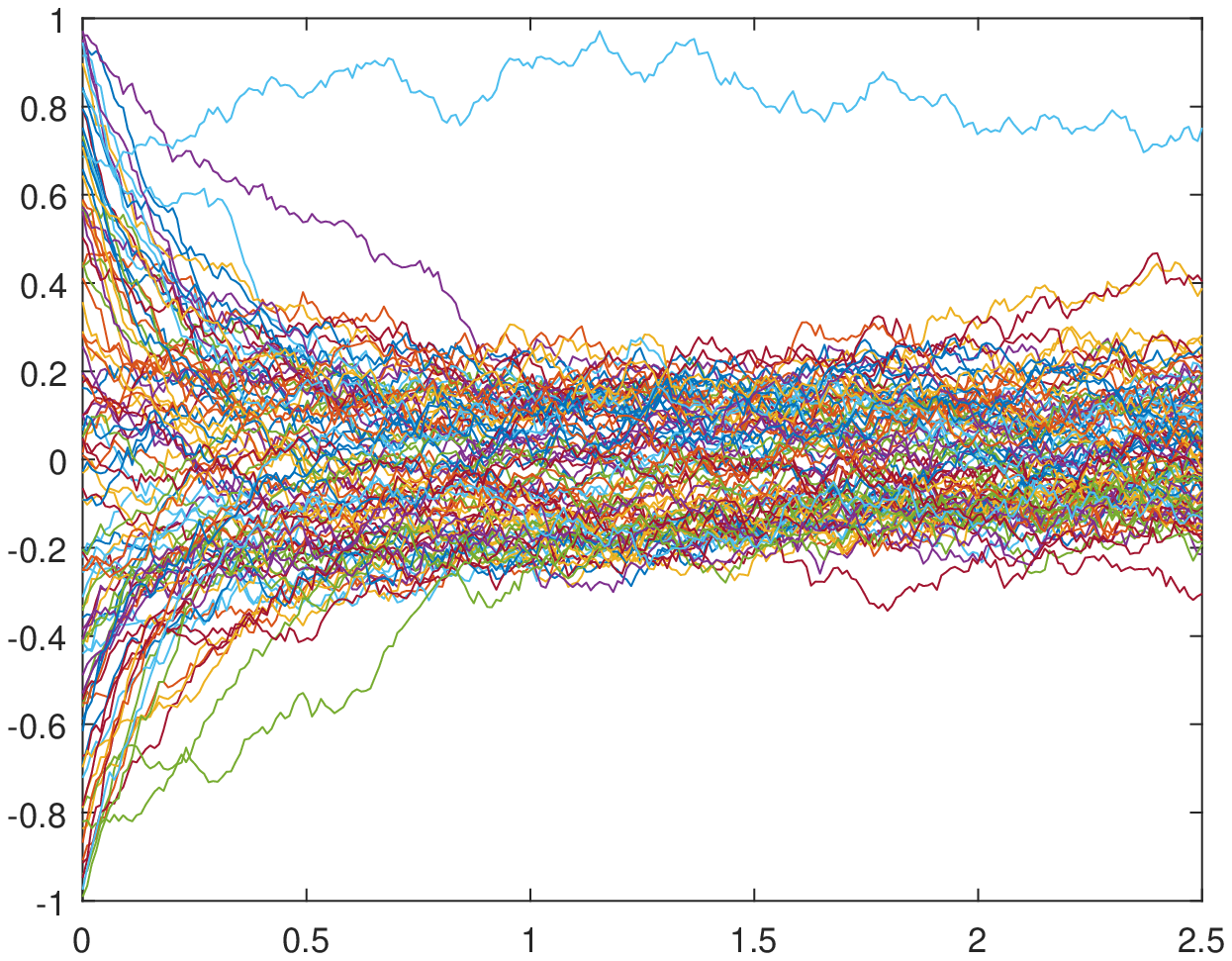}		         \caption{$\sigma = 0.15$} \label{subfig:TrAddNoise1}
	\end{subfigure}
	\caption{Opinion trajectories of agents' over time-period $[0,2.5]$ for different influences of additive noise: (a) $\sigma = 0.01$, (b) $\sigma = 0.05$ and (c) $\sigma = 0.15$. Other parameters are fixed to $R=0.015$ and $\alpha = \beta = 20$.}
	\label{fig:TrajectoriesAdditive_example}
\end{figure}

\begin{figure}[htb!]
	\centering
	\begin{subfigure}[b]{0.3\textwidth}
    \centering
	\includegraphics[width=\textwidth]{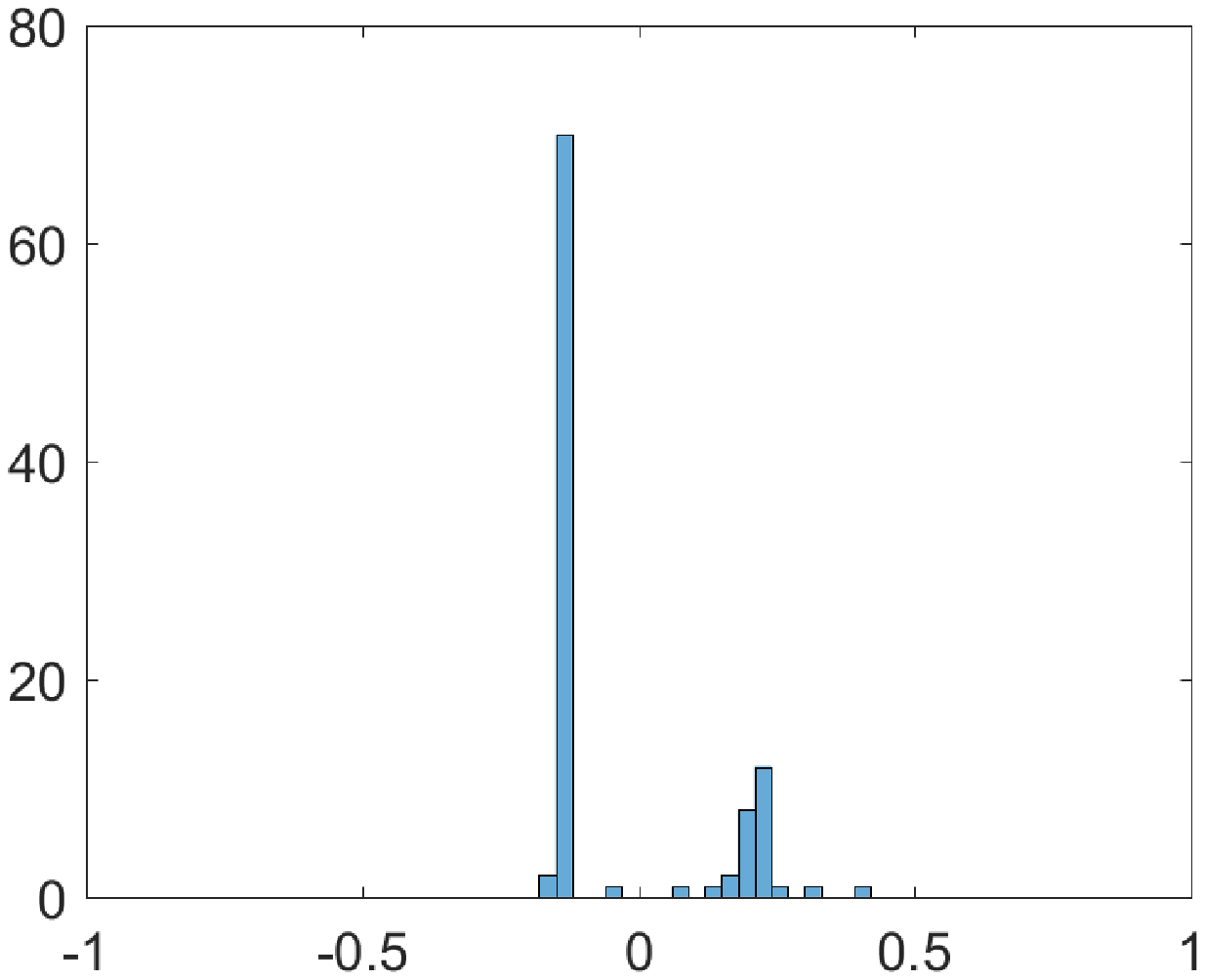}
	  \caption{$\sigma = 0.01$}
      \label{subfig:DAddNoise01}
     \end{subfigure}     
     \hfill
     \begin{subfigure}[b]{0.3\textwidth}
    \centering
	\includegraphics[width=\textwidth]{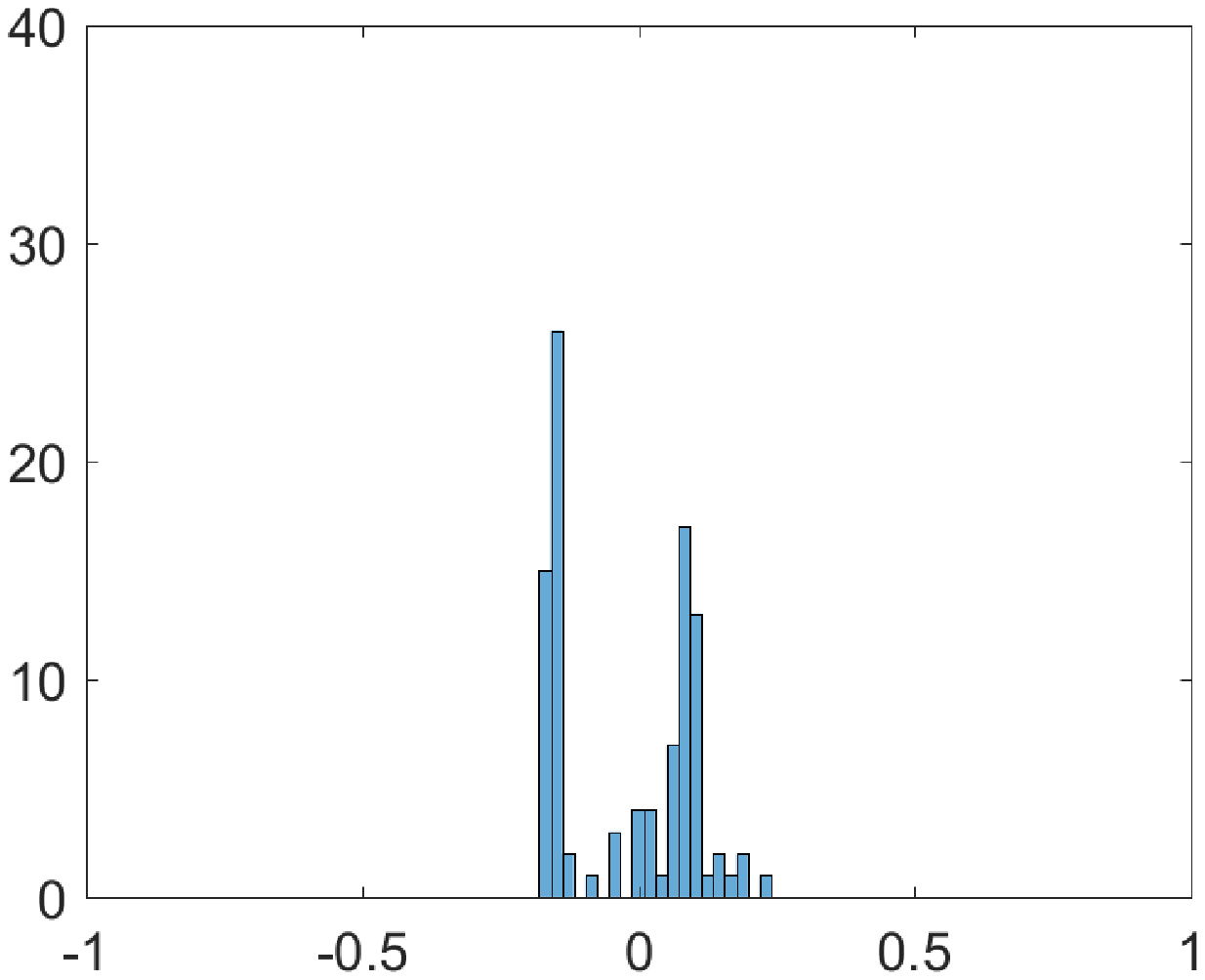}
		         \caption{$\sigma = 0.05$}
        \label{subfig:DAddNoise05}
     \end{subfigure}          
     \hfill
	\begin{subfigure}[b]{0.3\textwidth}
    \centering
	\includegraphics[width=\textwidth]{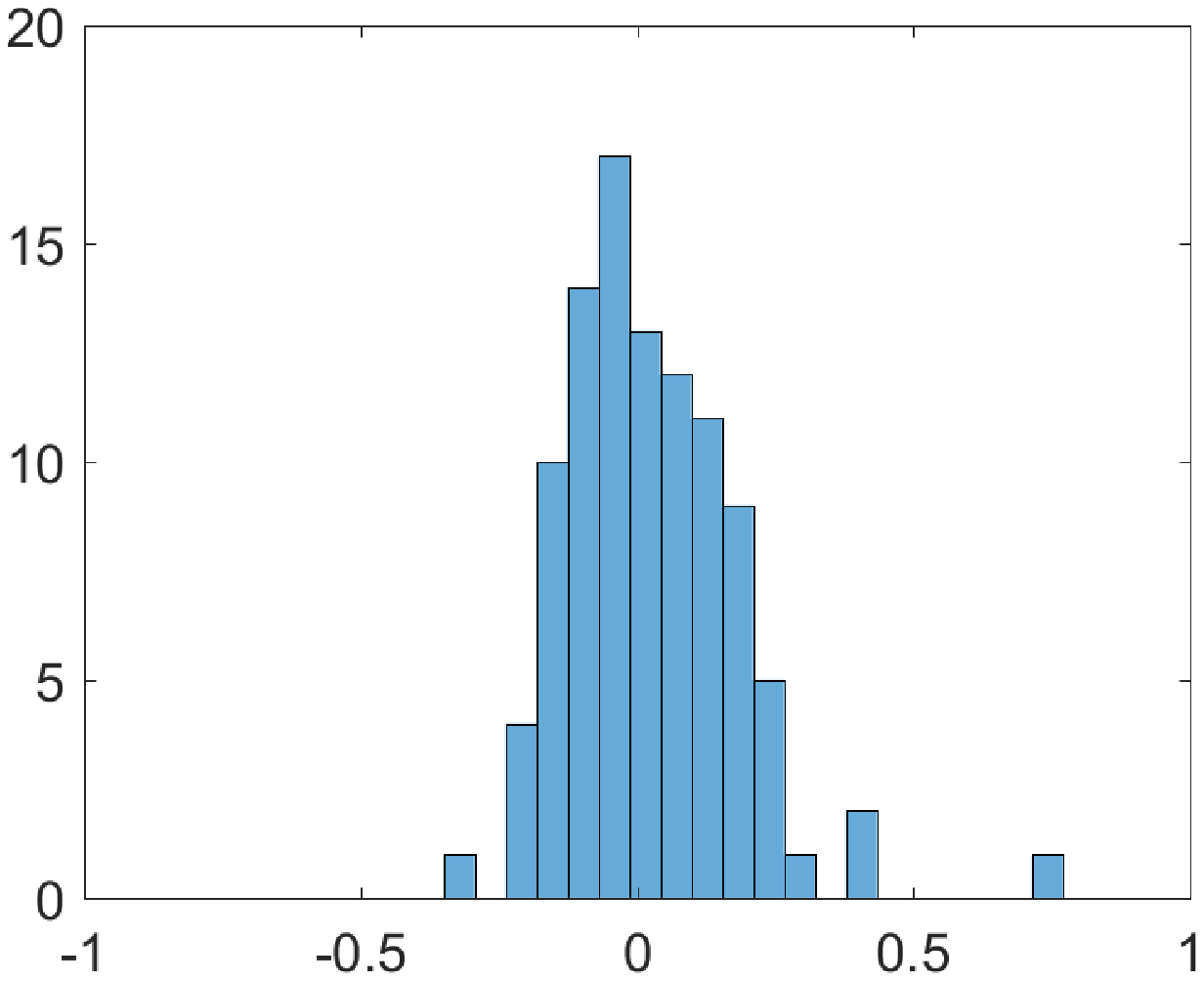}		         \caption{$\sigma = 0.15$} \label{subfig:DAddNoise1}
	\end{subfigure}
	\caption{Distribution of agents' opinions at final time $T=2.5$ for different influences of additive noise: (a) $\sigma = 0.01$, (b) $\sigma = 0.05$ and (c) $\sigma = 0.15$. Other parameters are fixed to $R=0.15$ and $\alpha = \beta = 20$.}
	\label{fig:DensityAdditive_example}
\end{figure}
We run stochastic simulations for $N = 100$ agents, where at $t=0$ all agents are placed uniformly at random inside $[-0.25,0.25]^2$ and their initial opinions are distributed uniformly in the interval $[-1,1]$. We run each simulation for $250$ time-steps, i.e. until $T = 2.5$ where $\Delta t = 0.01$. In every time-step, movement of agents and their opinions are obtained using a Euler–Maruyama scheme \cite{kloeden1992higher,higham2001algorithmic}.\\ 
First, we consider the case of additive noise, where we assume that $\sigma:=\sigma_{op} = \sigma_{sp}$ and we distinguish between different noise strengths, namely $\sigma = 0.01$, $\sigma = 0.05$ and $\sigma = 0.15$. We fix the interaction radius $R=0.15$ and include strong opinion influence $\alpha = 20$ and strong spatial influence $\beta =20$. In Figure \ref{fig:additive_example} we plot simulation snapshots at the final time $T = 2.5$ for these different values of $\sigma$, where agents' positions correspond to their position in a social space and agents' colours indicate their opinion. We observe a strong influence of noise on the cluster formation process. Namely, the higher the $\sigma$, the denser the inter-cluster connections are, such that for $\sigma = 0.15$ there is no clear separation into different clusters neither in social nor in opinion space. This can be seen in Figure \ref{fig:TrajectoriesAdditive_example} of opinion trajectories of individual agents. Additionally, stronger noise in the system leads to more diversified opinion distributions when reaching consensus, as can be observed in the final distribution of agents' opinions in Figure \ref{fig:DensityAdditive_example}. For small values of noise, i.e. $\sigma = 0.01$, the system reaches a stable state and we see several spatially separated clusters of agents that are stable and within each cluster agents have similar opinions. This behavior has been observed in previous studies of network models \cite{Yu2017,Royal2021}.\\
\begin{figure}[htb!]
	\centering
	\begin{subfigure}[b]{0.3\textwidth}
    \centering
	\includegraphics[width=\textwidth]{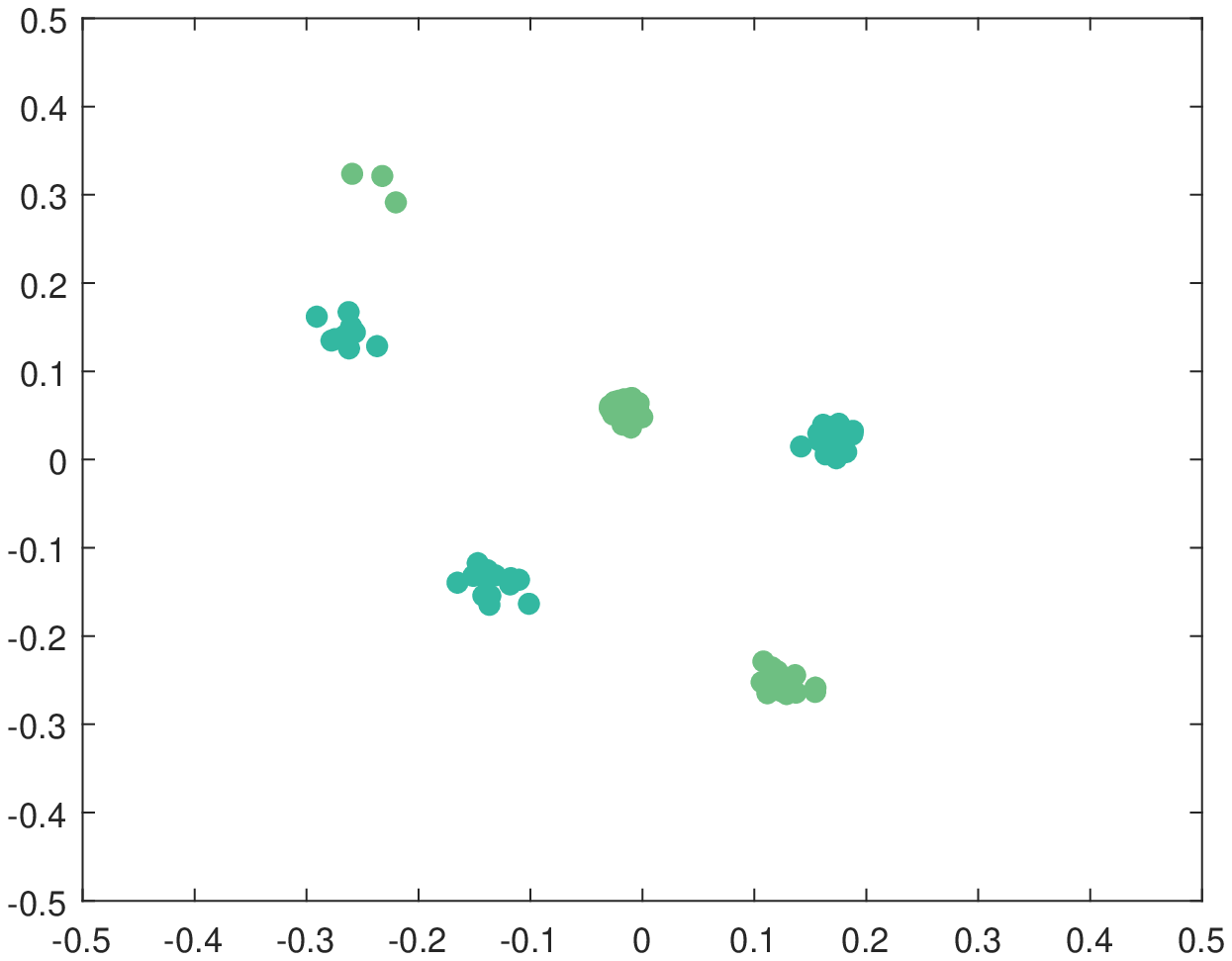}
	  \caption{$\alpha = \beta = 50$}
      \label{subfig:Alpha50AddNoise01}
     \end{subfigure}     
     \hfill
     \begin{subfigure}[b]{0.3\textwidth}
    \centering
	\includegraphics[width=\textwidth]{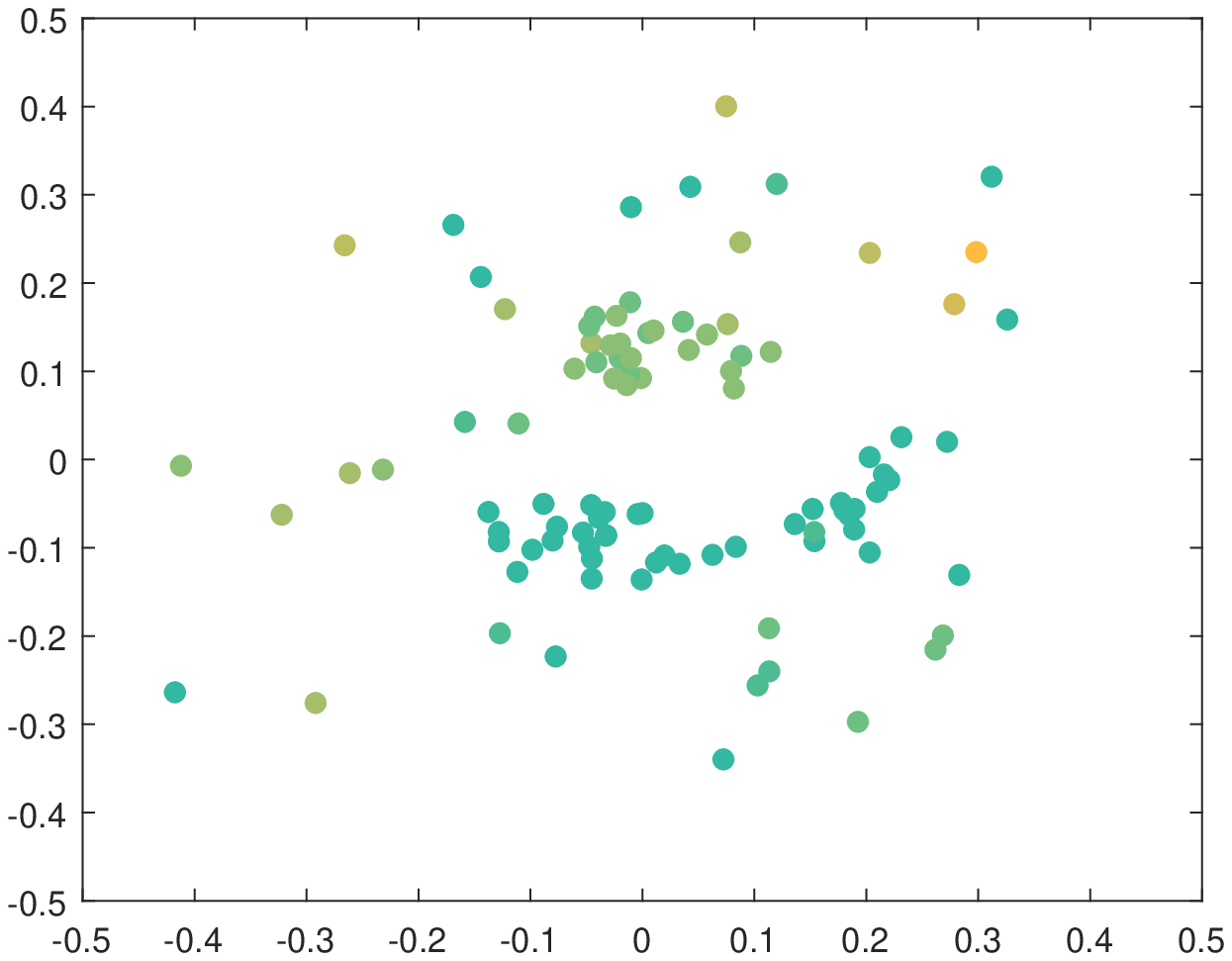}
	\caption{$\alpha = \beta = 5$}
    \label{subfig:Alpha5Beta5AddNoise05}
     \end{subfigure}
          \hfill
    \begin{subfigure}[b]{0.3\textwidth}
    \centering
	\includegraphics[width=\textwidth]{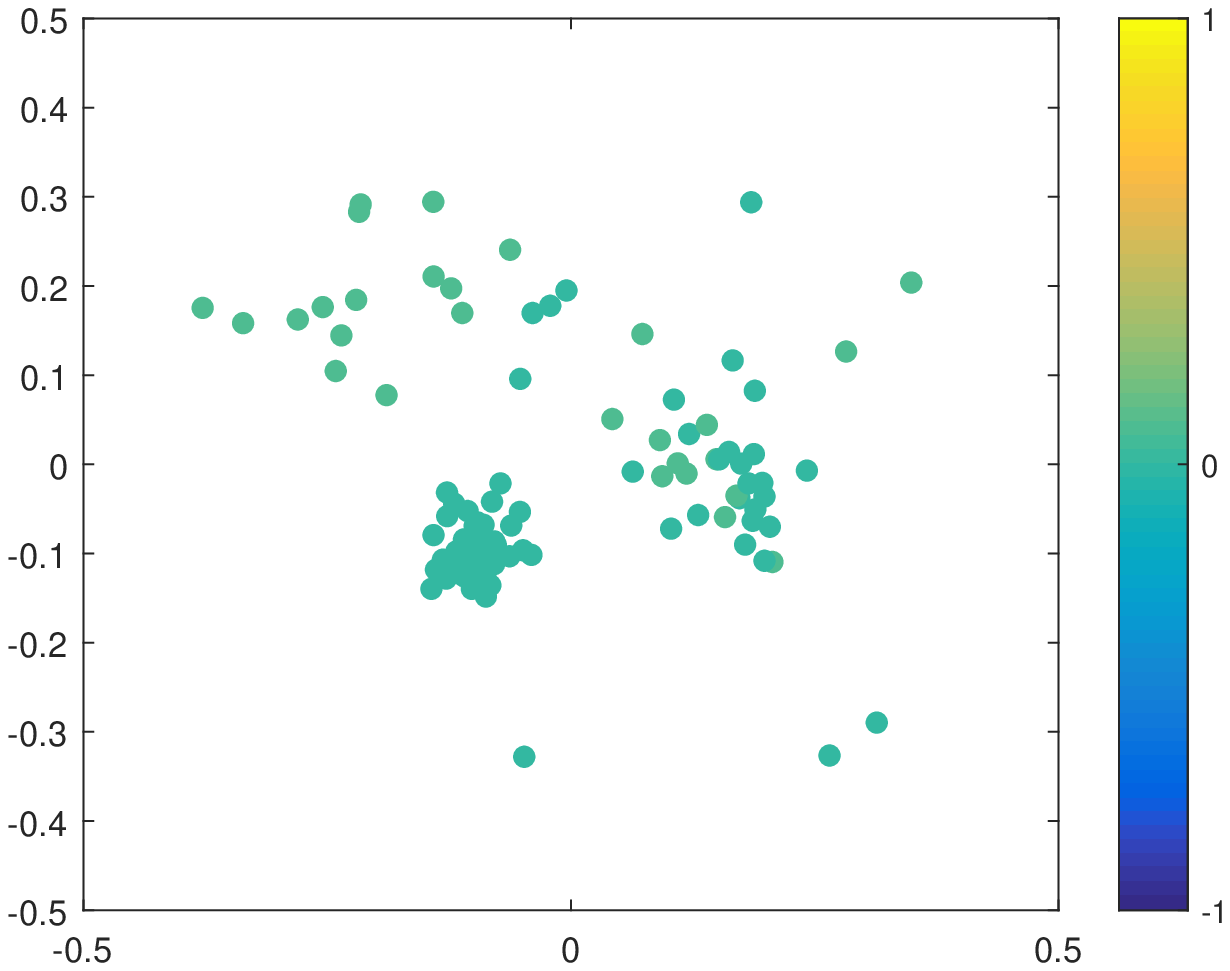}
	\caption{$\alpha = 50, \beta = 5$}
    \label{subfig:Alpha50AddNoise05}
     \end{subfigure}
	\caption{Snapshots from numerical simulations at final time $T=2.5$ for different influences of opinion and spatial strength: (a) $\alpha = \beta = 50$, (b) $\alpha = \beta = 5$, (c) $\alpha = 50, \beta = 5$. Position of agents indicated their position in a social space. Color of agents denotes their opinions according to the color-bar. Other parameters are fixed to $R=0.15$ and $\sigma = 0.05$.}
	\label{fig:OtherAlphaAdditive_example}
\end{figure}

\begin{figure}[htb!]
	\centering
	\begin{subfigure}[b]{0.3\textwidth}
    \centering
	\includegraphics[width=\textwidth]{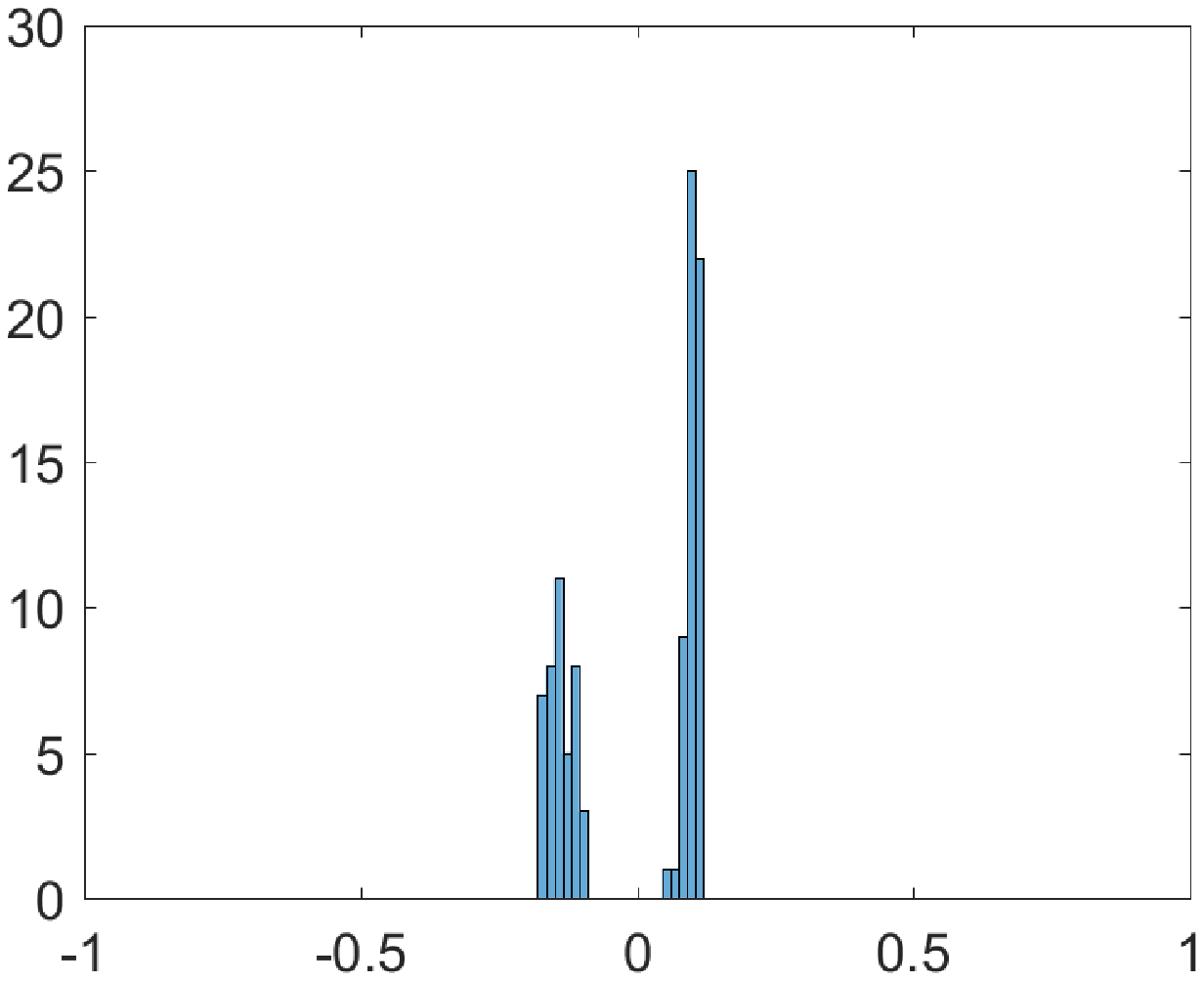}
	  \caption{$\alpha = 50$}
     \label{subfig:DAlpha50AddNoise01}
     \end{subfigure}     
     \hfill
     \begin{subfigure}[b]{0.3\textwidth}
    \centering
	\includegraphics[width=\textwidth]{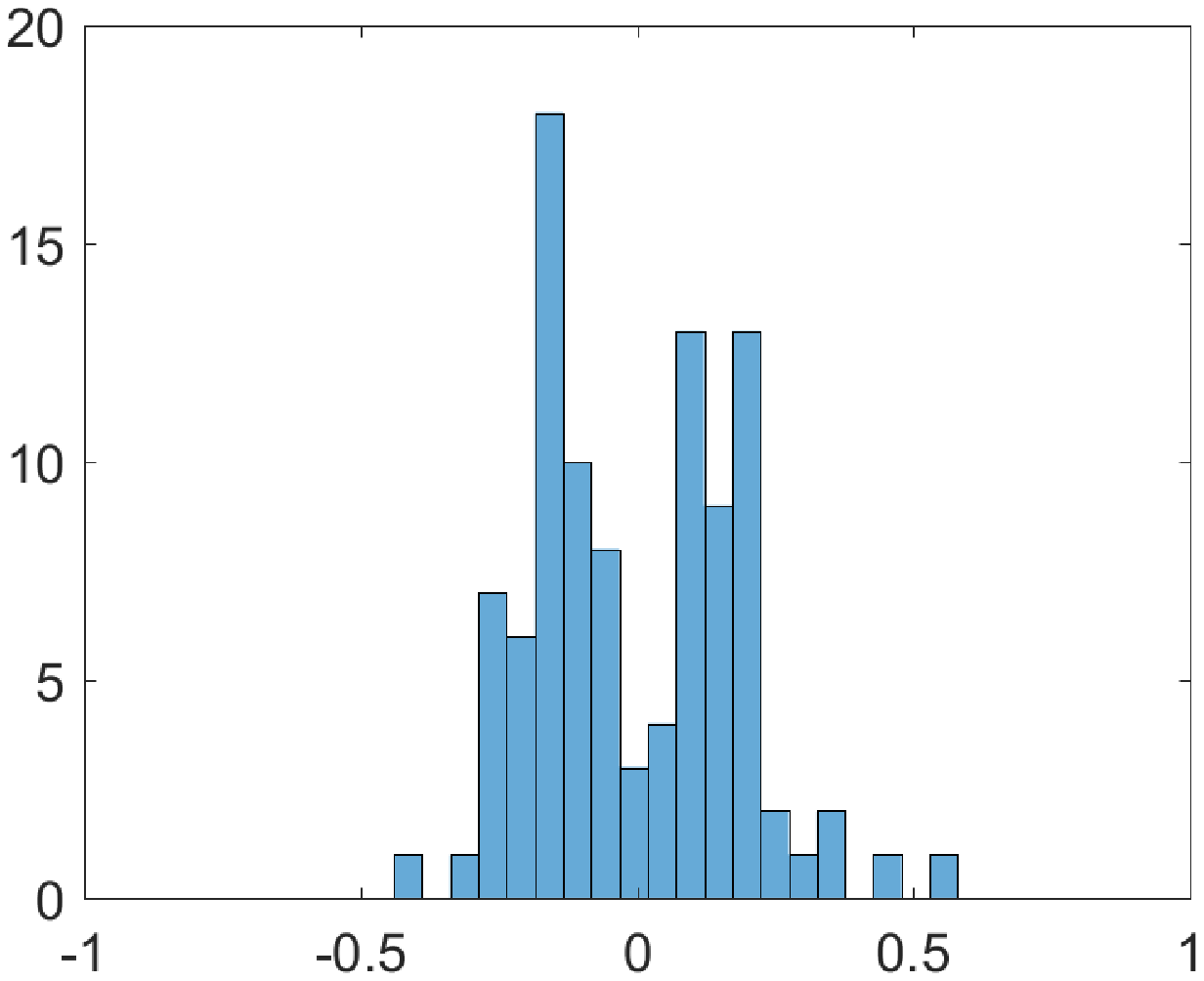}
	\caption{$\alpha = 5$}
    \label{subfig:DAlpha5AddNoise05}
     \end{subfigure}
    \hfill
     \begin{subfigure}[b]{0.3\textwidth}
    \centering
	\includegraphics[width=\textwidth]{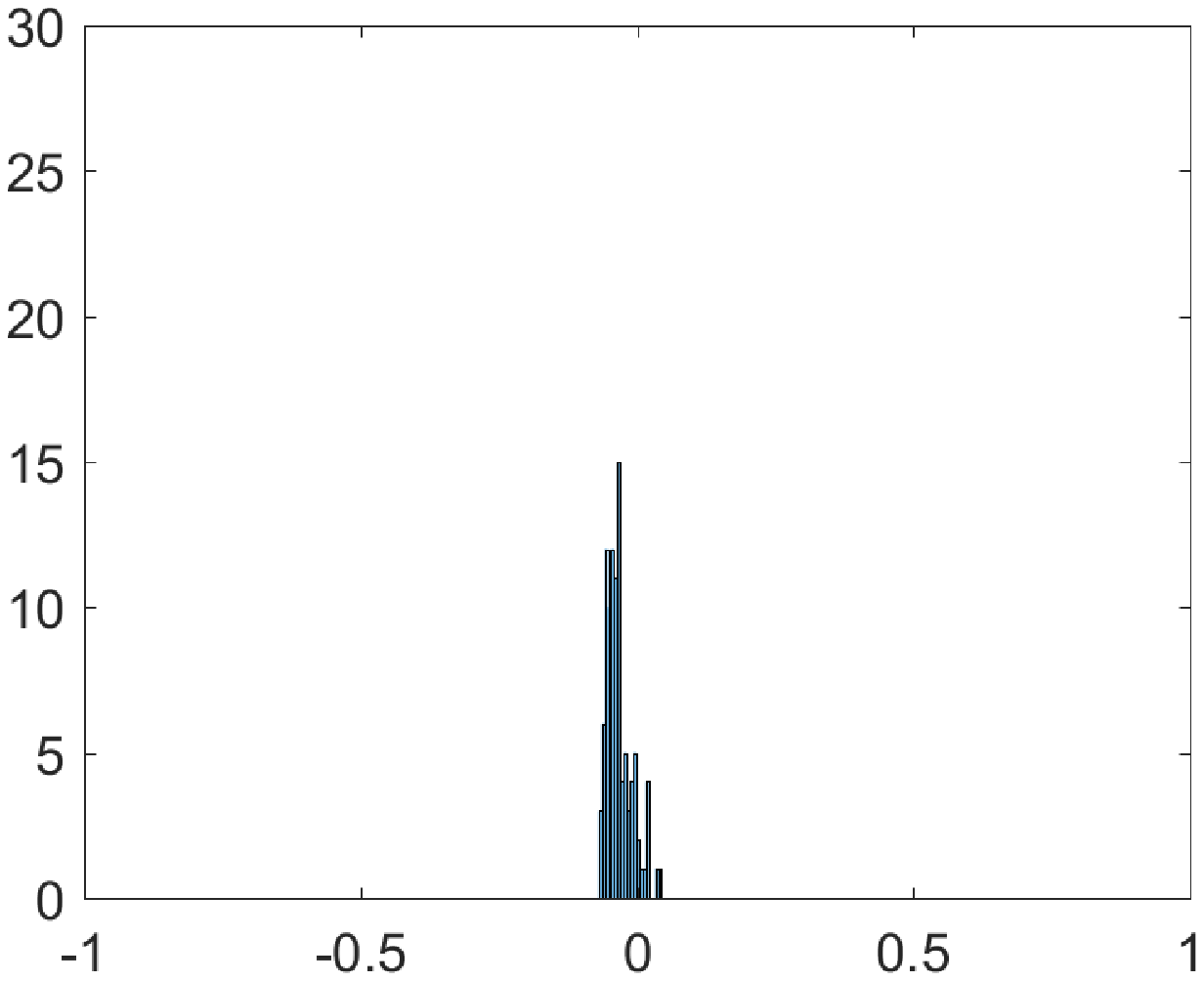}
	\caption{$\alpha = 50, \beta = 5$}
    \label{subfig:DAlpha50Beta5AddNoise05}
     \end{subfigure}\\
	\caption{Distribution of agents' opinions at final time $T=2.5$ for different influences of opinion and spatial strength: (a) $\alpha = \beta = 50$, (b) $\alpha = \beta = 5$. Position of agents indicated their position in a social space. Color of agents denotes their opinions according to the color-bar. Other parameters are fixed to $R=0.15$ and $\sigma = 0.05$.}
	\label{fig:DensityOtherAlphaAdditive_example}
\end{figure}
Next, we study the influence of opinion strength $\alpha$ and spatial strength $\beta$ on the cluster formation process. Additional to results shown in Figure\ref{subfig:AddNoise05} for $\alpha = \beta =20$, in Figure \ref{fig:OtherAlphaAdditive_example} we show the snapshots after $250$ time-steps for $\alpha = \beta= 50$, $\alpha = \beta = 5$ and $\alpha = 50, \beta = 5$. Large values of $\alpha, \beta$ mean that agents are strongly influenced by their peers, such that they compromise heavily towards their neighbours when updating their opinions and positions. This leads to faster formation of stable clusters where within the clusters agents have similar opinions, see Figure \ref{fig:OtherAlphaAdditive_example} and Figure \ref{fig:DensityOtherAlphaAdditive_example}. Because of this effect, in standard bounded confidence models, these parameters are often called ‘convergence parameters’ as they affect the speed of convergence \cite{PorterAdaptive2021}. Comparing the three scenarios in this experiment and in particular the cases when $\alpha=\beta$ to $\alpha\neq\beta$, we see that the spatial strength is important for formation of spatial clusters, but also in the case when spatial strength is small $\beta = 5$ and opinion strength is large $\alpha =50$, because of the feedback loop agents tend to form large, loose groups. This effect emphasises the importance of the feedback loop in this system.
\begin{figure}[htb!]
	\centering
	\includegraphics[width=0.3\textwidth]{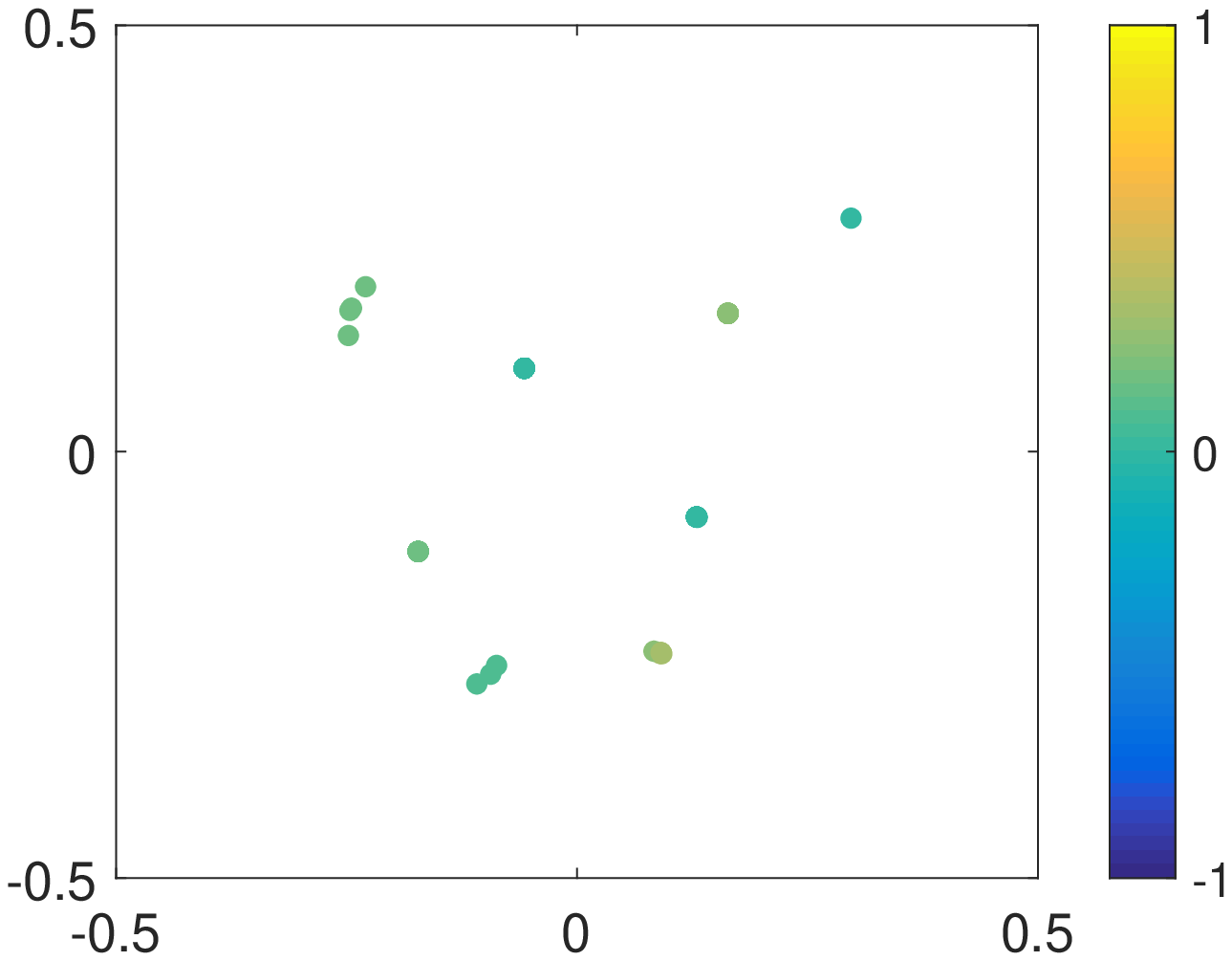}
	\includegraphics[width=0.3\textwidth]{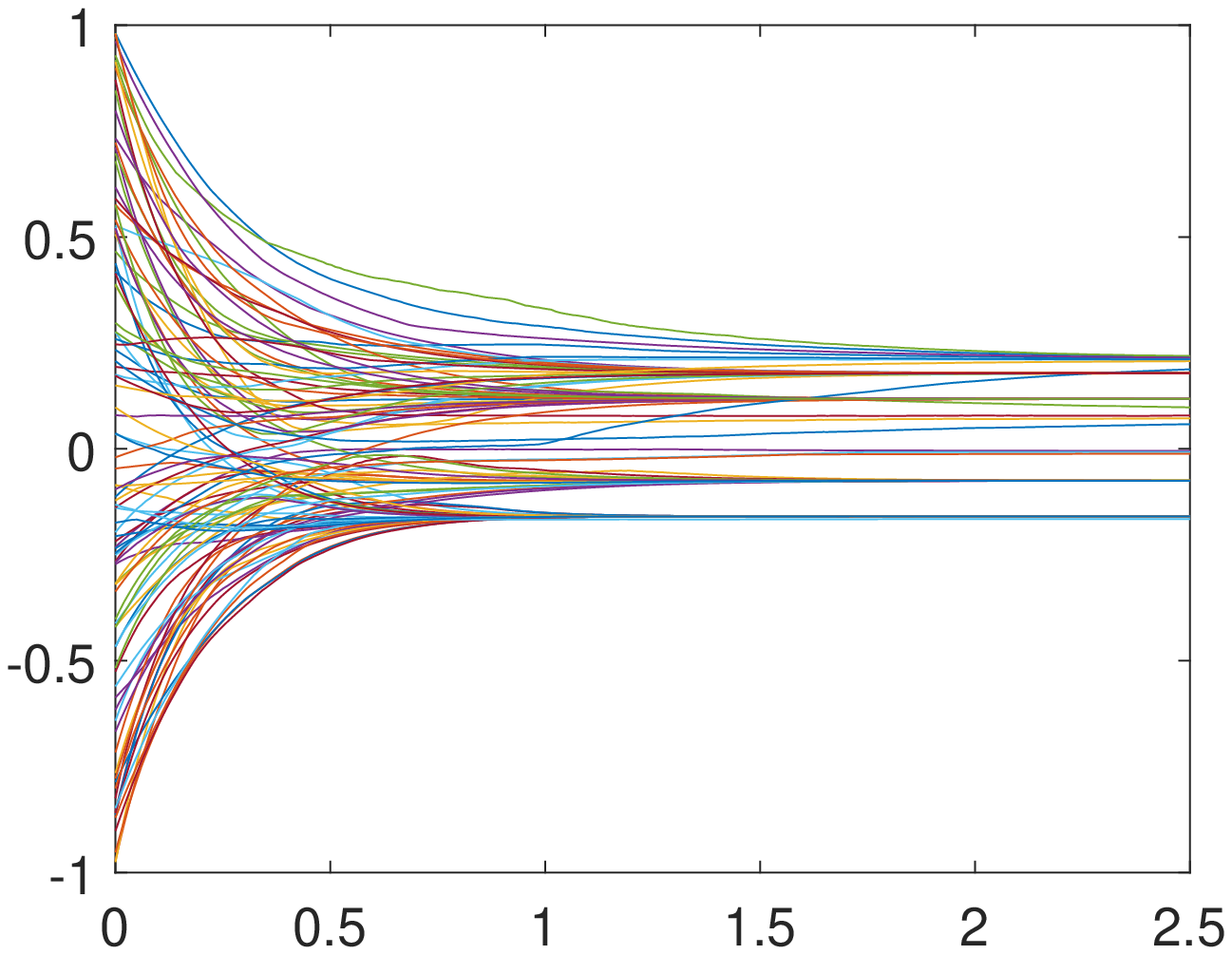}
	\includegraphics[width=0.3\textwidth]{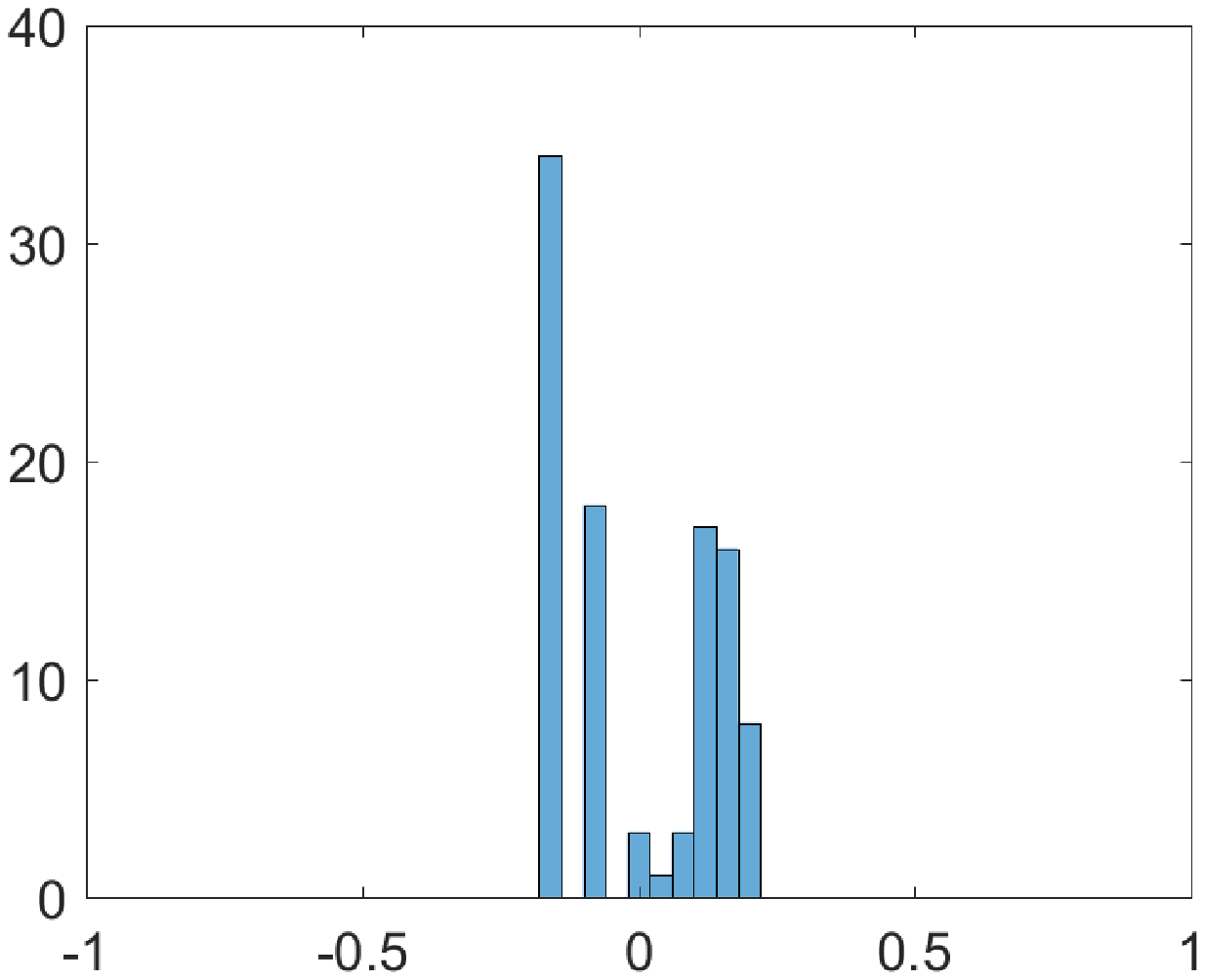}
	\caption{Results of one simulation of the ABM with multiplicative noise, for $\alpha = \beta = 20$ and $R = 0.15$. Left: Snapshot of the dynamics at $T=2.5$. Middle: Opinion trajectories during $[0,2.5]$. Right: Distribution of agents' opinions at $T=2.5$. Position of agents indicated their position in a social space. Color of agents denotes their opinions according to the color-bar. Other parameters are fixed to $R=0.015$ and $\sigma = 0.05$.}
	\label{fig:MultMultNoiseExample}
\end{figure}

Finally, we examine the effects of multiplicative noise on the dynamics of our ABM given by \eqref{eq:genSystem}, \eqref{eq:OpModel} and \eqref{eq:SpModel}. To this end, we consider multiplicative noise $\sigma_{op} = \sigma_{sp}$, as introduced in \eqref{eq:MultNoise}. In Figure \ref{fig:MultMultNoiseExample} we see a fast formation of $8$ spatially well separated clusters. In each of these clusters local consensus is reached, as discussed in Section \ref{sec:MultNoise}. These clusters are stable in their spatial position and opinion distribution. In comparison to the case of additive noise, see Figure  \ref{fig:additive_example}, clusters obtained through the influence of multiplicative noise are constructed faster and are more stable.

\section{Theoretical analysis: Coupled mean-field limit}\label{sec:MeanField}

In this section we consider the theoretical setting that describes the feedback loop dynamics of the system and its limit when $N \to \infty$. In particular, we start by briefly motivating the mean-field equation, and then we state its well-posedness and   the convergence of ABM to this mean-field equation.

From the theoretical point of view, one of the main challenges of these results lies in the regularity assumptions of drift and diffusion coefficients. The standard result is considering Lipschitz coefficients and additive noise, as presented in \cite{sznitman_topics_1991}. On the one hand, one can consider the dynamics just for one agent with singular interactions and additive noise, such as in \cite{krylov2005strong} and its extensions to multiplicative noise case as in \cite{zhang2016stochastic}. However, as noted in \cite{hao2022strong}, these results can not be trivially extended for $N$ particle systems. Thus, different techniques need to be used to obtain results for the singular interactions, such as mixed $L^p$ drifts. The literature on the weaker assumptions on drift and diffusion than the standard global Lipschitz assumption is tremendous as it is a very popular topic, see \cite{jabin2016mean, dos2022simulation,hammersley2021mckean, bresch2020mean, lacker2018strong}.

However, in order to illustrate the main message of this article that is the consideration of the feedback loop dynamics with the multiplicative noise, it is enough to consider the standard Lipschitz setting. Since the case of the multiplicative noise in this setting can not be found in an easily accessible form, for the completeness of results, we state in the appendix the proofs of the well-posedness and convergence results, without relying on the sophisticated analysis results.  This result can be seen as an extension of the standard result in \cite{sznitman_topics_1991}.

In order to simplify the notation, we will write $\int $ for $ \int_{\R^d \times \R} = \int_{\R^{d+1}} $ and we define $\mathcal{C} := C([0,T]; \R^{d+1})$. Hence, for the rest of the article, we will assume that   potentials $U$ and $V$ are real valued Lipschitz functions on $\R^{2d} \times \R^2$. More precisely, we will always assume  the following  on the regularity of initial data:

\begin{assumption}\label{IDassumption}
Let  $(B^{sp}, B^{op})$ be a Brownian motion in $\R^d\times \R$ with respect to a filtration $\calF = (\calF_t)_{t \geq 0}$, the initial conditions $X_0$ and $\Theta_0$ are $\calF_0$-measurable and square-integrable, and the maps 
\begin{align}
    &U:\R^d\times \R^d \times \R \times \R \to \R^d, 
    \quad 
    V:\R^d\times \R^d \times \R \times \R \to \R, 
    \\\
    &\sigma_{sp}: \R^d\times \R^d \times \R \times \R \to \R^{d \times d}, 
    \quad 
    \sigma_{op}: \R^d\times \R^d \times \R \times \R \to \R
\end{align}
are bounded and Lipschitz continuous. 
\end{assumption}

\subsection{Motivation for the limiting equations}

In this section we will present a general motivation on how one can derive the so-called mean-field limit of coupled SDEs on $[0,T]$ for a fixed number $N\in \N$ of interacting agents which is a special case of our ABM \eqref{eq:genSystem}. The mean-field limit is the coupled system of SDEs that describes the averaged dynamics of the system when the number of agents tends to infinity.
 
 In particular, from now on the coupled SDE system that we will analyse is given by 
\begin{align}\label{eq:agents}
     \begin{cases}
    dX^{k,N}_t 
    &= 
    \frac{1}{N}\sum_{j=1}^N U(X^k_t, X^j_t, \Theta^k_t, \Theta^j_t)  dt 
    + 
    \frac{1}{N}\sum_{j=1}^N \sigma_{sp}(X^k_t, X^j_t, \Theta^k_t, \Theta^j_t) dB^{sp,k}_t, 
    \\\
    d\Theta^{k,N}_t 
    &=
    \frac{1}{N}\sum_{j=1}^N V(X^k_t, X^j_t, \Theta^k_t, \Theta^j_t) dt + \frac{1}{N}\sum_{j=1}^N\sigma_{op}(X^k_t, X^j_t, \Theta^k_t, \Theta^j_t) dB^{op,k}_t,
    \end{cases}
\end{align}
for $k=1,...,N$, an independent family $(B^{sp,k}, B^{op,k})_{k \in \N}$ of $(d+1)$-dimensional Brownian motions.

Let us define the empirical density by 
\begin{equation}\label{def:empirical_measure}
 \mu_t^N := \frac{1}{N}\sum_{j=1}^N \delta_{(X^{j,N}_t, \Theta^{j,N}_t)}
\end{equation}
with $\delta$ being the Dirac distribution. Now the previous system can alternatively be written as
\begin{align}
    \begin{cases}
    dX^{k,N}_t 
    &= 
    \int_{\R^d \times \R}U(X^k_t, y, \Theta^k_t, \theta) \mu^N_t(dy d\theta) dt 
    + 
    \int_{\R^d \times \R} \sigma_{sp}(X^k_t, y, \Theta^k_t, \theta)\mu^N_t(dy d\theta) dB^{sp,k}_t,
    \\\
    d\Theta^{k,N}_t 
    &= \int_{\R^d \times \R} V(X^k_t, y, \Theta^k_t, \theta) \mu^N_t(dy d\theta) dt 
    +
    \int_{\R^d \times \R}\sigma_{op}(X^k_t, y, \Theta^k_t, \theta)\mu^N_t(dy d\theta) dB^{op,k}_t, 
    \\\
    \mu_t^N
    &= \frac{1}{N}\sum_{j=1}^N \delta_{(X^{j,N}_t, \Theta^{j,N}_t)}. 
    \end{cases}
\end{align}
Utilizing the law of large numbers we will show that the mean-field limit of the system is of the form
\begin{align}\label{lim_system}
    \begin{cases}
    d\overline{X}_t &= \left(\int_{\R^d \times \R}  U(\overline{X}_t, y, \overline{\Theta}_t, \theta) \mu_t(dy d\theta) \right)dt +  \left(\int_{\R^d \times \R}\sigma_{sp}(\overline{X}_t, y, \overline{\Theta}_t, \theta)\mu_t(dyd\theta)\right)dB^{sp}_t, \\\
    d\overline{\Theta}_t &= \left(\int_{\R^d \times \R} V(\overline{X}_t, y, \overline{\Theta}_t, \theta) \mu_t(dy d\theta) \right)dt +  \left(\int_{\R^d \times \R}\sigma_{op}(\overline{X}_t, y, \overline{\Theta}_t, \theta)\mu_t(dyd\theta)\right) dB^{op}_t,\\\
    \overline{\mu}_t &= \text{Law}(\overline{X}_t, \overline{\Theta}_t), \\\
    \overline{X}_{t=0} &= X_0, \, \overline{\Theta}_{t=0}=\Theta_0,
    \end{cases}
\end{align}
for suitable initial values $X_0, \Theta_0$ and a $(d+1)$-dimensional Brownian motion $(B^{sp}_t, B^{op}_t)$. 

\begin{Remark}
One can write the system more compactly by summarizing the state of each agent into a single $\R^{d+1}$-valued random variable $Y = (X, \Theta)$. In order to do this,  we define the combined interaction map to be 
\begin{align*}
    W: \R^d \times \R \times \R^d \times \R \to \R^{d+1}, \quad (x_1, \theta_1, x_2, \theta_2) \mapsto (U(x_1, x_2, \theta_1,\theta_2), V(x_1, x_2, \theta_1, \theta_2)).
\end{align*}
With this notation we can write the system \eqref{lim_system} as
\begin{align}\label{reduced-system}
    \begin{cases}
    dY_t &= \left(\int_{\R^{d+1}}W(Y_t, y, \theta) \mu_t(d(y,\theta)) \right)dt + \left( 
    \int_{\R^{d+1}}\Tilde{\sigma}(Y_t, y, \theta) \mu_t (d(y, \theta)) \right) d\Tilde{B}_t, \\\
    \mu_t &= \text{Law}(Y_t),
    \end{cases}
\end{align}
where $\Tilde{B}$ is a $(d+1)$-dimensional standard Brownian motion and $\tilde{\sigma}$ is a $(d+1)\times (d+1)$ diagonal matrix with $\text{diag}(\Tilde{\sigma_{sp}(X_t, \Theta_t)}) = (\sigma_{sp}(X_t, \Theta_t), ..., \sigma_{sp}(X_t, \Theta_t), \sigma_{op}(X_t, \Theta_t))$. 
\newline 
In the case of additive noise, if $U,V$ are assumed to be Lipschitz continuous, then $W$ is also a Lipschitz function and one can apply the classical existence and uniqueness and convergence results from \cite{sznitman_topics_1991} directly to $(\ref{reduced-system})$.   With the focus on the different interpretation of processes $X_t$ and $\Theta_t$ from the application point of view, we will keep the separate notion which enables us to directly see the limiting equations for position and opinion.
\end{Remark}

\subsection{Well-posedness result of the coupled mean-field SDE}

Here, we will study the well-posedness of the limiting stochastic differential equations given by \eqref{lim_system}.

The proof  is based on the  classical results on mean-field theory (see e.g. for additive noise \cite{sznitman_topics_1991}) and  the classical existence and uniqueness theory for SDEs via the standard fixed-point argument.

In order to set up our fixed point argument we, let us define 
\begin{align*}
    \calP^2
    &:=
    \calP^2(C([0,T]; \R^{d+1}))
    \\\
    &:= 
    \left\{ 
        \mu: \mu \text{ is a probability measure on } C([0,T];\R^{d+1})
        \text{ s.t. } \int{\norm{x}_{\infty}^2\mu(dx) < \infty}. 
    \right\}. 
\end{align*}
 Furthermore, we need to equip this space with a notion of distance, that turns $\calP^2$ into a complete metric space. 
To make use of a common tool for this type of proof, namely Gronwall's inequality, we need to consider a whole family of distances (which are not necessarily metrics due to a lack of definiteness). To be precise we define the \textit{truncated 2-Wasserstein distance} for $t \in [0,T]$  by 
\begin{align}\label{2Wass_def}
    D_t^2(\mu_1, \mu_2) := \inf_{\mu}
    \int_{\mathcal{C} \times \mathcal{C}} \left(\sup_{s \leq t} (\norm{X_s(\omega_1) - X_s(\omega_2)}^2 + \abs{\Theta_s(\omega_1) - \Theta_s(\omega_2)}^2) \right)d\mu(\omega_1, \omega_2),
\end{align}
where we take the infimum over all couplings  $\mu$ of $\mu_1$ and $\mu_2$ and $X_s(\cdot), \Theta_s(\cdot)$ are the projections onto the time $s$ marginal of the $d$-dimensional component and the $1$-dimensional component respectively.
Note that for $t=T$, i.e. if we don't truncate, we obtain the standard $2$-Wasserstein metric $\calW_2^2(\cdot, \cdot)$ on $\calP^2$.

\begin{Remark}\label{rem:Lipschitz}
    Observe that the Lipschitz Assumption \ref{IDassumption} on the coefficients also implies, that the induced maps defined by 
    \begin{align}
        \hat{U}: \R^d \times \R \times \calP^2(\R^{d+1}) \to \R^{d}, 
        \quad 
        (x, \theta, \mu) 
        \mapsto 
        \int_{\R^{d+1}}U(x, y, \theta, \eta)\mu(dy d\theta),
    \end{align}
   and analogously  $\hat{V}, \hat{\sigma}_{sp}, \hat{\sigma}_{op}$, 
    satisfy a Lipschitz-type inequality (\ref{wasserstein-lipschitz-inequality}) with respect to the product metric, when we equip $\calP^2$ with the $2$-Wasserstein metric.  
    Indeed, for $(x_1, \theta_1, \mu_1), (x_2, \theta_2, \mu_2)$ we have by Jensen's inequality
    \begin{align*}
        \norm{\hat{U}(x_1, \theta_1, \mu_1)\!\! -\!\! \hat{U}(x_2, \theta_2, \mu_2)}^2 \!\!
        \leq \!\!
        \int
        \norm{U(x_1, y_1, \theta_1, \eta_1)\! -\! U(x_2, y_2, \theta_2, \eta_2) }^2
        \pi(d(y_1, \theta_1)d(y_2, \theta_2)), 
    \end{align*}
    where $\pi$ is an arbitrary coupling of $\mu_1, \mu_2$. We can now use the Lipschitz assumption on $U$ to get 
    \begin{align*}
        &\int
        \norm{U(x_1, y_1, \theta_1, \eta_1) - U(x_2, y_2, \theta_2, \eta_2) }^2
        \pi(d(y_1, \theta_1)d(y_2, \theta_2)) 
        \\\
        \leq 
        &2L^2\norm{x_1 - x_2} + 2L^2\norm{\theta_1 - \theta_2} +
        2L^2\int
        \left(\norm{y_1 - y_2}^2 + \norm{\eta_1 - \eta_2}^2\right) \pi(d(y_1, \theta_1)d(y_2, \theta_2)).
    \end{align*}
    Now we can take the infimum over all couplings $\pi$ of $\mu_1$ and $\mu_2$ on the right hand side to obtain 
    \begin{align*}
        \label{wasserstein-lipschitz-inequality}
        &\norm{
        \int_{\R^{d+1}}U(x_1,y_1, \theta_1, \eta_1)\mu_1(dy_1 d\eta_1) 
        - 
        \int_{\R^{d+1}}U(x_2,y_2, \theta_2, \eta_2)\mu_2(dy_2 d\eta_2) 
        }^2
        \\\
        \leq
        &2L^2(\norm{x_1-x_2}^2 + \norm{\theta_1 - \theta_2}^2+\calW^2_2(\mu_1, \mu_2)). 
        \nonumber 
    \end{align*}
\end{Remark}

As already announced, to show the well-posedness of \eqref{lim_system}, we rely on Banach's classical fixed-point theorem. Therefore we need to make sure that the metric space $(\calP^2, D^2_T)$ is sufficiently regular. 

\begin{Lemma}\label{Lemma:completeMetricSpace}
$(\calP^2, D^2_T)$ is a complete metric space. 
\end{Lemma}

The proof of this lemma is standard and is given in Appendix A.1.

In addition, for the proof of the next theorem we need the following a priori estimate.
\begin{Lemma}\label{lemma:a-priori-estimate}
Assume that the coefficients $U,V, \sigma_{op}, \sigma_{sp}$ are bounded by some positive constant $K>0$ and that the initial values $X_0, \Theta_0$ satisfy $\mathbb{E}[\norm{X_0}^2]<\infty$ and $\mathbb{E}[\abs{\Theta_0}^2] < \infty$. Then for any $T>0$ there exists a constant $C = C(T,K)>0$ such that every solution $(X,\Theta)$ to (\ref{lim_system}) satisfies
\begin{align*}
    \mathbb{E}\left[\sup_{t \in [0,T]}\norm{X_t}^2+\abs{\Theta_t}^2\right] \leq C \left(\mathbb{E}\left[\norm{X_0}^2+\abs{\Theta_0}^2\right] + K^2(T^2+ T)\right),
\end{align*}
In particular, it holds that $Law(X, \Theta) \in \calP^2$. 
\end{Lemma}

The proof of this lemma is classic and it is based on the Burkholder-Davis-Gundy inequality. For the completeness we wrote the proof in Appendix A.1. Now we can state the well-posedness result:

\begin{Theorem}\label{theorem:existence-uniqueness}
Let $X_0$ and $\Theta_0$ be $\R^d$-valued, respectively $\R$-valued, random variables with finite second moment i.e. 
\begin{align*}
    \mathbb{E}\left[\norm{X_0}^2\right] < \infty, \quad \mathbb{E}\left[\abs{\Theta_0}^2\right] < \infty. 
\end{align*}
Under Assumption \ref{IDassumption}, there exists a unique (pathwise and in law) solution to equation $\eqref{lim_system}$.
\end{Theorem}

As already announced, the proof of the existence is based on the fixed point argument and it can be found in a more general setting. Nevertheless, for the completeness we sketch in the Appendix 1.1 the basic idea of the proof in our setting that is simpler than in the general setting, and hence easier accessible.

\subsection{Convergence of the microscopic model to the mean-field equation}

In this section we will prove that the 
 the system of coupled SDEs for fixed $N \in \mathbb{N}$ given by \eqref{eq:agents} indeed converges to the mean-field limit, i.e. to show the so called propagation of chaos. For $N \in \N$ we denote the (measure-valued) empirical process of this system by $\mu^N= (\mu^N_t)_{t \geq 0}$, i.e. for $t\geq 0$ we set
\begin{align}
    \mu_t^N := \frac{1}{N}\sum_{i=1}^N \delta_{(X^{i,N}_t, \Theta^{i,N}_t)}.
\end{align}

For fixed $i$ we consider the process $(\overline{X}^i_t, \overline{\Theta}^i_t)_{t \geq 0}$ that solves 
\begin{align*}
    \begin{cases}
    d\overline{X}^i_t 
    &= 
    \left(\int_{\R^{d+1}}
        U(\overline{X}^i_t, y, \overline{\Theta}^i_t, \theta) 
    \mu_t(dy \ d\theta) 
    \right) 
    dt
    +
    \left(
        \int_{\R^{d+1}}
            \sigma_{sp}(\overline{X}^i_t, y, \overline{\Theta}^i_t, \theta) 
        \mu_t(dy \ d\theta)   
    \right)
    dB^{i, sp}_t
    , 
    \\\
    d\overline{\Theta}^i_t 
    &=
    \left(
        \int_{\R^{d+1}}
            V(\overline{X}^i_t, y, \overline{\Theta}^i_t,\theta)
        \mu_t(dy \ d\theta)
    \right)
    dt
    +
    \left(
        \int_{\R^{d+1}}
            \sigma_{op}(\overline{X}^i_t, y, \overline{\Theta}^i_t, \theta) 
        \mu_t(dy \ d\theta)  
    \right)
    dB^{i, op}_t,
    \\\
    {\mu}_t 
    &=
    \text{Law}(\overline{X}_t^i, \overline{\Theta}_t^i). 
    \end{cases}
\end{align*}

The existence of $(\overline{X}^i_t, \overline{\Theta}^i_t)_{t \geq 0}$ follows from Theorem \ref{theorem:existence-uniqueness} and ${\mu} = ({\mu}_t)_{t \geq 0}$ does not depend on $i$. From now on we will denote the law of $(\overline{X}^i_t, \overline{\Theta}^i_t)_{t \geq 0}$ by ${\mu}$.

Next, we will study the case when $N$ tends to $\infty$. The key idea for the proof is to use the law of large numbers (LLN) for empirical measures $\overline{\mu}^N$ of i.i.d. copies of the process $(\overline{X}^i, \overline{\Theta}^i)$, cf. \cite[Theorem 11.4.1]{dudley_real_2002}. We will see that by a uniform integrability argument, LLN also implies that 
\begin{align*}
    \mathbb{E}\left[D^2_T({\mu}, \overline{\mu}^N)\right] \to 0 \quad \text{as } N \to \infty. 
\end{align*}
The rest of the needed estimates are then the same as in the proof of Theorem \ref{theorem:existence-uniqueness} and are mainly done for the purpose of showing that 
\begin{align*}
    \mathbb{E}\left[D_T^2({\mu}, \mu^N)\right]
    \leq 
    C \mathbb{E}\left[D^2_T({\mu}, \overline{\mu}^N)\right]
\end{align*}
for some constant $C>0$. 
\begin{Theorem}\label{theorem:convergence-microscopic-model}
For any $i \in \N$ we have 
\begin{align*}
    \mathbb{E}\left[
    \sup_{0\leq s \leq t}\norm{X^{i,N}_s - \overline{X}^i_s}^2 + \abs{\Theta^{i,N}_s - \overline{\Theta}^i_s}^2
    \right] \to 0 \quad \text{as } N \to \infty. 
\end{align*}
Moreover, we  have 
\begin{align}
    \lim_{N \to \infty} \mathbb{E}\left[D_T^2({\mu}, \mu^N)\right] =0.
\end{align}
\end{Theorem}

The details of the proof are presented in Appendix 1.2.

\section{Characterization of the empirical measure and its limit}\label{sec:EmpiricalMeasure}
In this section we derive the formal stochastic partial differential equation (SPDE) for the empirical measure $\mu^N$, that is the so-called Dean--Kawasaki type equation \cite{dean_langevin_1996, kawasaki1994stochastic}  with multiplicative noise and we derive the McKean-Vlasov type PDE  for the limiting measure $\mu$. Recall that if
 $X = (X_t)_{t \geq 0}$ is a Markov process with generator $\mathcal{L}$, then its law $\mu = (\mu_t)_{t \geq 0}$ is a (weak) solution to the linear PDE \begin{align}
    \partial_t \mu = \mathcal{L}^*\mu, 
\end{align}
also known as the Kolmogorov forward equation. A similar equation still holds true for solutions $(X, \Theta, \mu)$of the type of equation \eqref{lim_system} and empirical measure for \eqref{eq:agents}. However,  the PDE that we will derive will be non-linear.

\subsection{Derivation of the PDE for the law of the coupled mean-field SDEs}

We consider a system of coupled mean-field SDEs given by \eqref{lim_system} and we want to derive equation for the hydrodynamic limit $ \mu_t = \text{Law}(X_t, \Theta_t)$. 

Let  $\phi: \R^d \times \R \to \R$ be  a smooth and compactly supported function.
Then by Itô's formula we have 
\begin{align*}
    & d\phi(X_t, \Theta_t) 
    = \\
    &\!\!\!\left(
        \int U(X_t, y, \Theta_t, \theta) \cdot \nabla_x \phi(X_t, \Theta_t) \mu_t(d y d \theta)
    \right) 
    dt 
    +
    \left(
        \int V(X_t, y, \Theta_t, \theta) \frac{\partial}{\partial \theta}\phi(X_t, \Theta_t) \mu_t(d y d \theta)
    \right) 
    dt \\\
    + 
    &\frac{1}{2}\Bigg(\!\!\sum_{i,j=1}^d \left(\int\sigma_{sp}\!(X_t, y,\Theta_t, \theta)\mu_t(dy d\theta)\! \cdot \!\left(\!\int\! \sigma_{sp}(X_t,y, \Theta_t,\theta)\mu_t(dy d\theta)\right)\!\!^\top \right)_{i,j} \!\!\!\!
    \times \frac{\partial^2}{\partial x_i \partial x_j}\phi(X_t, \Theta_t)\!\!\Bigg) dt 
    \\\
    + 
    &\frac{1}{2}\left(\left(\int \sigma_{op}(X_t,y, \Theta_t, \theta)\mu_t(dy d\theta)\right)^2 \frac{\partial^2}{\partial \theta^2}\phi(X_t, \Theta_t) \right)dt
    \\\
    +&\left((\nabla_x \phi(X_t, \Theta_t))\!\!^\top \cdot \int \sigma_{sp}(X_t,y, \Theta_t, \theta)\mu_t(dy d\theta)  ) \right)dB^{sp}_t
    \\\
    +&\left(\int \sigma_{op}(X_t,y, \Theta_t, \theta)\mu_t(dy d\theta) \cdot \frac{\partial}{\partial \theta} \phi(X_t, \Theta_t)\right) dB^{op}_t . 
\end{align*}

Taking the expectation the martingale part vanishes,  differentiating in $t$ and recalling that $\mu_t = \text{Law}(X_t, \Theta_t)$ we can rewrite the previous equation  as

\begin{align*}
   & \frac{d}{dt}\int \phi(z,\eta) \mu_t(dz d\eta) 
    = \\\
    &\int \bigg[
        \left(
            \int U(z, y, \eta, \theta) \cdot \nabla_z \phi(z, \eta) \mu_t(d y d \theta)
        \right)
        +
        \left(
            \int V(z, y, \eta, \theta) \frac{\partial}{\partial \eta}\phi(z, \eta) \mu_t(d y d \theta)
        \right)
        \\\
        \quad+
        &\frac{1}{2}\Bigg(\sum_{i,j=1}^d \left(\int \sigma_{sp}(X_t, y,\Theta_t, \theta)\mu_t(dy d\theta) \cdot \left(\!\!\int \sigma_{sp}(X_t,y, \Theta_t,\theta)\mu_t(dy d\theta)\right)^\top \right)_{i,j} 
        \\\
        & \quad \times  \frac{\partial^2}{\partial x_i \partial x_j}\phi(X_t, \Theta_t)\Bigg)
        +
        \frac{1}{2}\left(\left(\int \sigma_{op}(X_t,y, \Theta_t, \theta)\mu_t(dy d\theta)\right)^2 \frac{\partial^2}{\partial \theta^2}\phi(X_t, \Theta_t) \right)
        \bigg]
    \mu_t(dz d\eta).
\end{align*}
Assuming that $\mu_t$ is absolutely continuous with respect to Lebesgue measure and that its density $\mu_t(z,\eta)$ is sufficiently regular,  we can apply Fubini's theorem and integration by parts. 
Since $\phi$ is an arbitrary sufficiently regular test-function,  we see that $\mu = (\mu_t)_{t \geq 0}$ is a distribution-valued solution of the PDE 
\begin{equation}\label{equation:MF_lim}
\begin{aligned}
    \partial_t \mu_t(z, \eta) 
    =
    &-\text{div}_z \left(\mu_t(z,\eta) \cdot U(z,\eta, \mu_t)\right)
    -\text{div}_{\eta}\left(\mu_t(z,\eta) \cdot V(z,\eta, \mu_t) \right)
    \\\
    &+ \frac{1}{2}\sum_{i,j=1}^d\frac{\partial^2}{\partial z_i \partial z_j}\left(\mu_t(z,\eta) \cdot \sigma_{sp}(z,\eta, \mu_t) \cdot \sigma_{sp}(z,\eta, \mu_t)^\top\right)_{i,j} 
    \\\
    +  &\frac{1}{2}
        \frac{\partial^2}{\partial \eta^2} \left( \mu_t(z, \eta) \sigma_{op}(z,\eta, \mu_t)^2\right), 
\end{aligned}
\end{equation}
where we use the shorthand notation
\begin{align}\label{short_notation}
    U(z,\eta, \mu_t) &:= \int_{\R^d \times \R}U(z, y, \eta, \theta)\mu_t(dy d\theta), 
    \\\
    V(z, \eta, \mu_t) &:= \int_{\R^d \times \R}V(z, y, \eta, \theta) \mu_t(dy d\theta), 
    \\\
    \sigma_{space}(z, \eta, \mu_t) &:= \int_{\R^d \times \R}\sigma_{space}(z,y, \eta, \theta) \mu_t(dy d\theta), 
    \\\
    \sigma_{op}(z, \eta, \mu_t) &:= \int_{\R^d \times \R}\sigma_{op}(z,y, \eta, \theta) \mu_t(dy d\theta). \label{short_notation_end}
\end{align}

\begin{Remark}
Note that in the case of the additive noise this equation becomes
\begin{align*}
\partial_t \mu_t(z,\eta) 
=&-\text{div}_z\Big( U(z,\eta, \mu_t) \mu_t(z,\eta)\Big)- \text{div}_{\eta}\Big(V(z,\eta,\mu_t) \mu_t(z, \eta)\Big) \\\ \nonumber
&+ \frac{1}{2}\sigma_1^2 \Delta_z \mu_t(z,\eta) + \frac{1}{2}\sigma_2^2\Delta_{\eta}\mu_t(z,\eta), 
\end{align*}
which coincides with the standard result from \cite{sznitman_topics_1991}.
Hence, as expected the only difference are the coefficients of the second order part of the differential operator on the right hand side, which are either constants or functions.
\end{Remark}

\subsection{SPDE description for the empirical measure $\mu^N$}

The goal is to derive a formal SPDE for the empirical measure $\mu^N$, i.e. the analogue to the
Dean-Kawasaki equation, but with multiplicative noise.
More precise, for $N \in  \N$ let $(X^N_t, \Theta_t^N)_{t\geq 0}$ be such that its components $(X^{k,N}_t, \Theta_t^{k,N})_{t\geq 0}$ solve the system \eqref{eq:agents}. 

In the following, we want to (formally) derive the SPDE that is solved by the empirical measure $\mu^N = (\mu^N_t)_{t\geq 0}$ defined by \eqref{def:empirical_measure}.

Let $\phi: \R^d \times \R \to \R$ be a smooth and compactly supported function. By definition of the empirical measure we have 
\begin{align*}
    \int_{\R^d \times \R} \phi(z, \eta) \mu_t^N(dz d\eta) = \frac{1}{N}\sum_{i=1}^N \phi(X^{i,N}_t, \Theta^{i,N}_t). 
\end{align*}
Hence, we will first calculate each of the summands individually. For fixed $i=1,\dots,N$ we denote by $B^{sp,i} = (B^{sp,i,1}, \dots, B^{sp,i,d})$ the components of the driving Brownian motion of the $i$-th agent. By Itô's formula we have 
\begin{align*}
    d\phi(X_t^{i,N}, \Theta_t^{i,N}) 
    &= 
    \int \mu^{i,N}_t(dz d\eta) \Bigg[
    \bigg(
        \int
        U(z,\eta,\mu_t^N) \cdot \nabla_x \phi(z, \eta)
    +   V(z, \eta, \mu^N_t) \frac{\partial}{\partial \theta}\phi(z, \eta)
    \bigg) 
    dt \\\
    + 
    &\frac{1}{2}\left(\sum_{i,j=1}^d \left(\sigma_{sp}(z, \eta, \mu_t^N) \cdot \sigma_{sp}(z, \eta, \mu_t^N)^\top\right)_{i,j} \frac{\partial^2}{\partial z_i \partial z_j}\phi(z, \eta)\right) dt 
    \\\
    + 
    &\frac{1}{2}\left(\sigma_{op}(z, \eta, \mu^N_t)^2 \frac{\partial^2}{\partial \eta^2}\phi(z, \eta) \right)dt
    \\\
    +&\left((\nabla_z \phi(z, \eta))^\top\sigma_{sp}(z, \eta,\mu_t^N) \right)dB^{sp,i}_t 
    +\left(\sigma_{op}(z, \eta, \mu_t^N)\frac{\partial}{\partial \theta} \phi(z, \eta)\right) dB^{op,i}_t
    \Bigg],
\end{align*}
where we again used the shorthand notation introduced in \eqref{short_notation}-\eqref{short_notation_end}.

Assuming that $\mu^{i,N}_t$ is absolutely continuous with respect to Lebesgue measure with a sufficiently smooth density $\mu_t^{i,N}(z, \eta)$ and applying integration by parts with respect to the Lebesgue integral, we obtain
\begin{align*}
    d\phi(X_t^{i,N}, &\Theta_t^{i,N}) 
    =
    \\\
    & \int dz d\eta 
        \phi(z, \eta)
        \Bigg[
        -\text{div}_z \left[U(z,\eta, \mu^N_t) \mu_t^{i,N}(z,\eta)\right]dt
        -\text{div}_\eta \left[V(z, \eta, \mu^N_t) \mu^{i,N}_t(z,\eta)\right]dt
        \\\
        +&\frac{1}{2}\left(\sum_{i,j=1}^N  \frac{\partial^2}{\partial z_i \partial z_j}\left(\left(\sigma_{sp}(z,\eta, \mu^N_t) \cdot \sigma_{sp}(z,\eta, \mu^N_t)^\top\right)_{i,j} \mu^{i,N}_t(z,\eta) \right) \right)dt
        \\\
        +&\frac{1}{2}\frac{\partial^2}{\partial \eta^2}\left(\sigma_{op}(z,\eta, \mu^N_t)^2 \mu^{i,N}(z,\eta)\right)dt
        -
        \sum_{k,l=1}^d \frac{\partial}{\partial z_k}\left(\sigma_{sp}^{l,k}(z,\eta, \mu^N_t)\mu^{i,N}(z,\eta)dB^{sp,i,k}_t \right)
        \\\
        -
        &\frac{\partial}{\partial \eta}\left(\sigma_{op}(z,\eta, \mu^N_t)\mu^{i,N}(z,\eta)dB^{op,i}_t \right)
        \Bigg].
\end{align*}

Utilizing that for each $i=1,\dots, N$ we have 
\begin{align*}
    \frac{d}{dt}\phi(X^{i,N}, \Theta^{i,N}) 
    =
    \int_{\R^{d+1}}dz d\eta \phi(z,\eta) \left[\frac{\partial}{\partial t} \mu^{i,N}_t(z, \eta)\right],
\end{align*}

and after the summation over $i$ we conclude that the empirical measure $(\mu^N_t)_{t\geq 0}$ solves the SPDE
\begin{align*}
    \partial_t \mu^N_t 
    =
    &-\text{div}_z\left[U(z,\eta, \mu^N_t) \mu^N_t(z,\eta) \right]
    -\text{div}_\eta \left[ V(z,\eta, \mu^N_t) \mu^N_t\right] 
    \\\
    &+ \frac{1}{2} \sum_{i,j=1}^N \left( \frac{\partial^2}{\partial z_i \partial z_j}\left(\sigma_{sp}(z,\eta, \mu^N_t) \cdot \sigma_{sp}(z,\eta, \mu^N_t)^\top\right)_{i,j} \mu^{N}_t(z,\eta) \right)
    \\\
    &+ \frac{1}{2}\frac{\partial^2}{\partial \eta^2}\left(\sigma_{op}(z,\eta, \mu^N_t)^2 \mu^N(z,\eta)\right)
    \\\
    &-\frac{1}{N}\sum_{i=1}^N\sum_{k=1}^q \sum_{l=1}^d \frac{\partial}{\partial z_k}\left(\sigma_{sp}^{l,k}(z,\eta,\mu^N_t)\mu^{i,N}_t(z,\eta) dB^{sp,i,k}_t\right)
    \\\
     &-\frac{1}{N}\sum_{i=1}^N\frac{\partial}{\partial \eta}\left(\sigma_{op}(z,\eta) \mu^{i,N}_t(z,\eta) dB^{op,i}_t \right).
\end{align*}
However, this is not yet a closed equation for the empirical measure, because it still depends on the individual trajectories through the noise term. 
As in the case of additive noise \cite{dean_langevin_1996}, we will first calculate the covariance of this noise term and then replace it by a statistically identical term. For this, let's denote the two noise terms by 
\begin{align*}
    \xi^{sp}(t,z,\eta) &:=
    \frac{1}{N}\sum_{i=1}^N\sum_{k=1}^d \sum_{l=1}^d \frac{\partial}{\partial z_k}\left(\sigma_{sp}^{l,k}(z,\eta,\mu^N_t)\mu^{i,N}_t(z,\eta) dB^{sp,i,k}_t\right),
    \\\
    \xi^{op}(t,z,\eta) &:=
    \frac{1}{N}\sum_{i=1}^N\frac{\partial}{\partial \eta}\left(\sigma_{op}(z,\eta) \mu^{i,N}_t(z,\eta) dB^{op,i}_t \right).
\end{align*}
Then, the covariances are given by 
\begin{align*}
    \bbE\left[\xi^{sp}(t,z,\eta)\xi^{sp}(s,y,\theta) \right] &=
    \delta(t-s) \frac{1}{N^2}
    \sum_{i=1}^N \sum_{k=1}^d \sum_{l,g=1}^d 
    \frac{\partial}{z_k} \left(\sigma_{sp}^{l,k}(z,\eta, \mu^N_t) \mu^{i,N}_t(z,\eta) \right) 
    \\\
    &\hspace{4cm}
    \times \frac{\partial}{x_k} \left( \sigma_{sp}^{g,k}(x,\theta, \mu^N_s) \mu^{i,N}_s(x,\theta)  \right)
    ,
    \\\
    \bbE\left[ \xi^{sp}(t,z,\eta) \xi^{op}(s,x,\theta)\right] &= 0
    ,
    \\\
    \bbE\left[ \xi^{op}(t,z,\eta) \xi^{op}(s,x,\theta)\right] &=
    \delta(t-s) \frac{1}{N^2}\sum_{i=1}^N \frac{\partial}{\partial \eta}  \left(\sigma_{op}(z,\eta, \mu^N_t)\mu^{i,N}_t(z,\eta)\right) 
    \\\
    &\hspace{3cm}
    \times \frac{\partial}{\partial \theta} \left(\sigma_{op}(x, \theta, \mu^N_s)\mu^{i,N}_s(x,\theta) \right). 
\end{align*}

First note that by the definition of $\mu^{i,N}_t$ as a Dirac delta distribution, we have for each $i=1,\dots,N$
\begin{align*}
    \mu^{i,N}_t(z,\eta) \mu^{i,N}_t(x, \theta) 
    =
    \delta(z-x)\delta(\eta - \theta)\mu^{i,N}(z,\eta)
    =
    \delta(z-x)\delta(\eta - \theta)\mu^{i,N}(x,\theta). 
\end{align*}
Therefore we can rewrite the non-trivial covariances as 
\begin{align*}
   &\bbE\left[\xi^{sp}(t,z,\eta)\xi^{sp}(s,y,\theta) \right]
     \\\
    &= 
    \delta(t-s)\frac{1}{N}
   \sum_{i=1}^N \sum_{k=1}^d \sum_{l,g=1}^d
        \frac{\partial}{\partial z_k}\frac{\partial}{\partial x_k}\left(\delta(\eta-\theta) \delta(z-x)\sigma^{l,k}_{sp}(z,\eta, \mu^N_t) \sigma_{sp}^{g,k}(z,\eta, \mu^N_t)\mu^{N}_t(z,\eta)\right)
   , 
 \end{align*}
 and
 \begin{align*}
        &\bbE\left[ \xi^{op}(t,z,\eta) \xi^{op}(s,x,\theta)\right] 
    &=
    \delta(t-s) \frac{1}{N}\frac{\partial}{\partial \eta}\frac{\partial}{\partial \theta}
    \left( \delta(z-x) \delta(\eta-\theta) \sigma_{op}(z,\eta, \mu^N_t)^2  \mu^N_t(z,\eta) \right).
\end{align*}
Furthermore, let $\zeta$ be a $(d+1)$-dimensional Gaussian random field indexed by $\R_+ \times \R^d \times \R$ with covariance given by
\begin{align*}
    \bbE \left[ \zeta^i(t,z,\eta) \zeta^j(s,x, \theta)\right]
    = \delta^{i,j}\delta(t-s) \delta(z-x) \delta(\eta-\theta).
\end{align*}
We define two noise fields  $\xi^{sp'},\xi^{op'}$ via 
\begin{align*}
   \xi^{sp'}(t,z,\eta) &:=
   \frac{1}{\sqrt{N}}\sum_{k=1}^d \sum_{l,g=1}^d
        \frac{\partial}{\partial z_k}\frac{\partial}{\partial x_k}\left(\sigma^{l,k}_{sp}(z,\eta, \mu^N_t) \sigma_{sp}^{g,k}(z,\eta, \mu^N_t) \mu^{N}_t(z,\eta)^{\frac{1}{2}}\zeta(t,z,\eta) \right)
   \\\
   \xi^{op'}(t,z,\eta) &:=
   \frac{1}{\sqrt{N}}\frac{\partial}{\partial \eta}\frac{\partial}{\partial \theta}\left(\sigma_{op}(z,\eta, \mu^N_t) \mu^N_t(z,\eta)^{\frac{1}{2}} \zeta(t, z,\eta) \ \right).
\end{align*}
Note that these noise fields are statistically equivalent to $\xi^{sp}, \xi^{op}$.
Altogether, we conclude that the dynamics of the empirical measure $\mu^N$ of the feedback-loop dynamics with multiplicative noise given by \eqref{eq:diffusion_process_all} is described by the following  formal SPDE
\begin{equation}\label{eq:MainSPDEN}
\begin{aligned}
    \partial_t \mu^N_t 
    =
    &-\text{div}_z\left[U(z,\eta, \mu^N_t) \mu^N_t(z,\eta) \right]
    -\text{div}_\eta \left[ V(z,\eta, \mu^N_t) \mu^N_t(z,\eta)\right] 
    \\\
    &+ \frac{1}{2} \sum_{i,j=1}^N \left( \frac{\partial^2}{\partial z_i \partial z_j}\left(\sigma_{sp}(z,\eta, \mu^N_t) \cdot \sigma_{sp}(z,\eta, \mu^N_t)^\top\right)_{i,j} \mu^{N}_t(z,\eta) \right)
    \\\
    &+ \frac{1}{2}\frac{\partial^2}{\partial \eta^2}\left(\sigma_{op}(z,\eta, \mu^N_t)^2 \mu^N(z,\eta)\right)
    \\\
    &-
    \frac{1}{\sqrt{N}}\sum_{k=1}^d \sum_{l,g=1}^d
        \frac{\partial}{\partial z_k}\frac{\partial}{\partial x_k}\left( \sigma^{l,k}_{sp}(z,\eta, \mu^N_t) \sigma_{sp}^{g,k}(z,\eta, \mu^N_t) \mu^{N}_t(z,\eta)^{\frac{1}{2}}\zeta^k(t,z,\eta) \right)
    \\\
    &-
    \frac{1}{\sqrt{N}}\frac{\partial}{\partial \eta}\frac{\partial}{\partial \theta}\left(\sigma_{op}(z,\eta, \mu^N_t) \mu^N_t(z,\eta)^{\frac{1}{2}} \zeta^{d+1}(t, z,\eta) \ \right).
\end{aligned}
\end{equation}

\begin{Remark}
The previous equation \eqref{eq:MainSPDEN} is the generalization of the standard Dean-Kawasaki equation \cite{dean_langevin_1996} to the dynamics with the multiplicative case. In particular, in the case of the additive noise this SPDE becomes
\begin{align*}
    \partial_t \mu^N_t(x, \theta) 
    = 
    &- \text{div}_x \left[\mu^{N}_t(x, \theta) \cdot \int U(x,y,\theta, \eta)\mu^N_t(y,\eta)dy d\eta \right] \nonumber 
    \\\
    &- \text{div}_{\theta}\left[\mu^{N}_t(x, \theta) \cdot \int V(x,y,\theta, \eta)\mu^N_t(y,\eta)dy d\eta \right] 
    \\\
    &+ \frac{1}{2}\sigma_1^2 \Delta_x \mu_t^{N}(x, \theta) + \frac{1}{2}\sigma_2^2 \Delta_{\theta}\mu^{N}_t(x, \theta) \nonumber  
    \\\
    &- \frac{1}{N}\sigma_1\text{div}_x\left[\mu^N_t(x, \theta)^{\frac{1}{2}} \cdot \zeta(t,x,\theta) \right]
    - \frac{1}{N}\sigma_2 \text{div}_{\theta}\left[\mu^N_t(x, \theta)^{\frac{1}{2}} \cdot \zeta(t,x,\theta) \right] \nonumber
\end{align*}
which is the standard type of Dean-Kawasaki equation. Concerning the SPDE formulations of an ABM, observe that the diffusion part that appears in the dynamics of the empirical measure (see \cite{HelfmannDjurdjevacConradDjurdjevacetal.2021}) is exactly of this type  and the case with multiplicative noise is its generalization. 
\end{Remark}

\subsection{Numerical experiment}

We will here illustrate the behaviour of the feedback loop system at the macroscopic level, i.e. when the number of agents tend to infinity. As we showed in the previous Section, this dynamics is described by the PDE \eqref{equation:MF_lim}. We will simulate it using finite difference method with one spatial dimension and one opinion dimension. The considered domain is $[-2,2]^2$ with the grid size $0.05$ and the time interval is $[0,1]$ with time step $dt=0.0001$. We use no-flux boundary conditions. The initial conditions are randomly chosen four clusters with  normal distribution. Other parameters are chosen in a similar way as in the experiments made in Section \ref{sec:Simulation}. In particular, the potentials $U$ and $V$ are given by  \eqref{eq:OpModel} and \eqref{eq:SpModel} respectively, with the additional scaling parameter that is taken to be $0.5$ and represents scaled space/opinion interaction strength. Note that these potentials do not satisfy our regularity Assumption \eqref{IDassumption}, however, as already explained, from the literature is expected that the PDE equation for the empirical density has the same form as \eqref{equation:MF_lim}. As in the ABM, the interaction radius is taken to be $0.15$ and we consider the additive noise with strength $0.01$ for both space and opinion dynamics. In Figure \ref{fig:PDEExample} we show the empirical density of agents from the numerical discretization of the equation \eqref{equation:MF_lim} at the initial time $t = 0$, intermediate time $t = 0.5$ and final time $t = 1$. We observe that the behavior of the ABM shown in \ref{sec:Simulation} agrees with the emerging dynamics of the PDE model. Namely, the agents empirical density show the cluster formation that is in agreement with the ABM and reflects the feedback loop dynamics of the system. Note that the diffusive behaviour depends on the choice of the scaled space/opinion strength parameter. In particular, it is expected to have stronger clustering effects with the increasing of space/opinion strength parameter. More detailed investigation of this, and the effect of the boundary conditions will be the topic for the future research. 

\begin{figure}[htb!]
	\centering
	\begin{subfigure}[b]{0.3\textwidth}
    \centering
	\includegraphics[width=\textwidth]{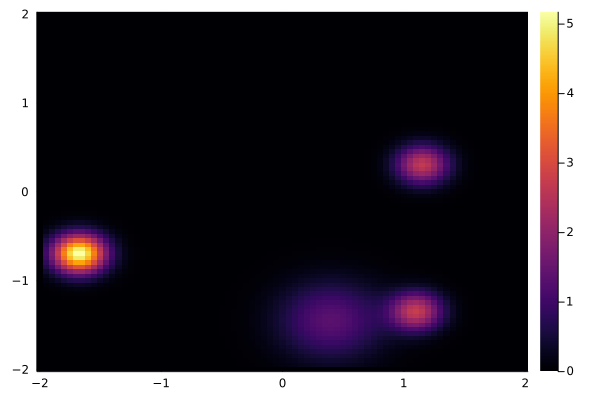}
	  \caption{$t=0$}
      \label{subfig:PDEInitTime}
     \end{subfigure}     
     \hfill
     \begin{subfigure}[b]{0.3\textwidth}
    \centering
	\includegraphics[width=\textwidth]{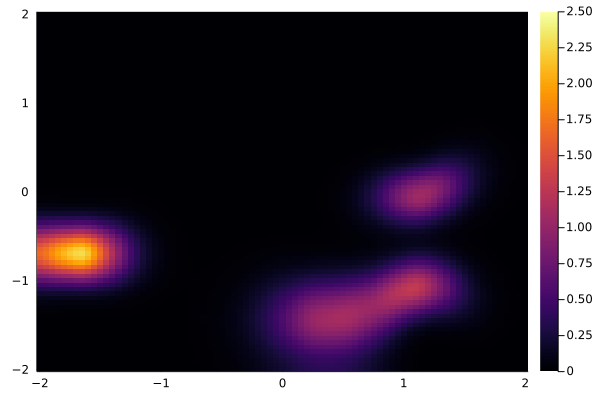}
	\caption{$t=0.5$}
    \label{subfig:PDEIntermediateTime}
     \end{subfigure}
          \hfill
    \begin{subfigure}[b]{0.3\textwidth}
    \centering
	\includegraphics[width=\textwidth]{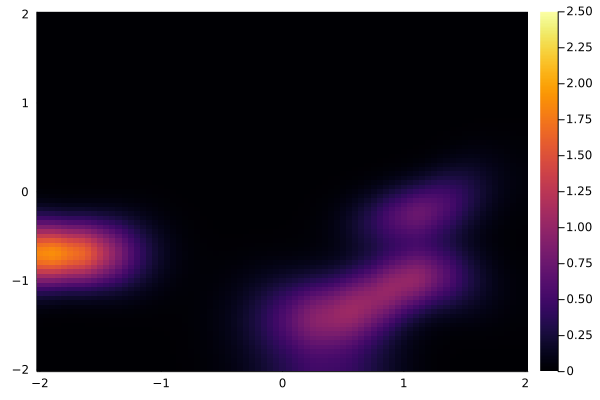}
	\caption{$t=1$}
    \label{subfig:PDEFinalTime}
     \end{subfigure}
	\caption{Empirical density of agents in mean-field limit given by equation \eqref{equation:MF_lim} at initial time $t=0$, intermediate time $t=0.5$ and final time $t=1$. 
	}
	\label{fig:PDEExample}
\end{figure}

\section{Conclusion}\label{sec:Discussion}
Literature on opinion dynamics is very rich and diverse, spanning from empirical approaches, that rely on real-world data, to mathematical models, that mostly consider simple social rules and allow for rigorous analysis. However, there is a large gap between data- and model- driven approaches, that asks for introducing new formal models that can better represent complex social mechanisms governing how people shape and share their opinions. 
As a step towards closing this gap, in this paper we introduced an agent-based model for studying the feedback loop between opinion and social dynamics. This co-evolving dynamics governs how agents position in a social space (e.g. online social media) is influencing and is being influenced by other agents opinions. Additionally, unlike most existing models for opinion dynamics that consider only deterministic dynamics or additive noise, in order to account for more realistic scenarios, we introduce the influence of a multiplicative noise. In order to explore how these different factors influence the appearance of emerging phenomena in the system, we tested our ABM model on several toy examples. In particular, we simulated the model for different parameter choices and we showed how these govern grouping of agents into spatial clusters, within which agents hold similar opinions. Our experiments have shown that cluster formation with respect to agents' social and/or opinion states is strongly influenced by the feedback loop. Further investigations of this model and possible effects the feedback loop could have in real-world social systems, will be the topic of future research.\\
Additionally, we formulated the feedback loop model in a rigorous mathematical framework and considered its behaviour in the case when the number of agents tends to infinity. Although the well-posedness propagation of chaos results have been proved for these type of systems in various weak regularity assumptions scenarios, we stated the proofs of these results in the case of the Lipschitz assumption that are easier to access and present the extension of the work presented in \cite{sznitman_topics_1991}. Finally, we derived formal equation for the empirical density of the ABM that is in the spirit of the so-called Dean-Kawasaki equation, and the equation of its hydrodynamic limit. Motivated by applications, considering the feedback loop  dynamics with more singular interaction potentials and its numerical analysis would be a natural next step and will be considered in the future.




\section*{Acknowledgments}
The authors would like to thank Christof Sch\"utte, Nicolas Perkowski and Luzie Helfmann for insightful discussions on ABM, McKean-Vlasov SDE and mean-field limit. We acknowledge the support of Deutsche Forschungsgemeinschaft (DFG) under the Germany's Excellence Strategy - The Berlin Mathematics
Research Center MATH+ (EXC-2046/1, project ID 390685689). A. Dj.  has been partially funded by Deutsche Forschungsgemeinschaft (DFG) through grant CRC
1114 "Scaling Cascades in Complex Systems", Project Number 235221301, Project C10.


\section[\appendixname~\thesection]{Appendix}
\subsection*{Appendix A.1.}
\begin{proof}[Proof of Lemma \ref{Lemma:completeMetricSpace}]
This follows from Theorem \cite[Theorem 6.18]{villani_optimal_2009}  and the fact that
\begin{align*}
    d((x,\theta), (y,\eta)) := \sup_{s\leq T} \norm{(x_s,\theta_s)-(y_s,\eta_s)}, \quad (x,\theta), (y,\eta) \in C([0,T];\R^{d+1}),
\end{align*}
where 
$     \norm{(x_s, \theta_s)} := \left(\sum_{k=1}^d \abs{x_s^k}^2 + \abs{\theta_s}^2 \right)^2,$
defines a metric on $C([0,T]; \R^{d+1})$. 
\end{proof}

\begin{proof}[Proof of Lemma \ref{lemma:a-priori-estimate}]
For $n \in \N$ consider the stopping time 
\begin{align*}
    \tau_n := \min\set{\inf\set{t \geq 0: \abs{X_t}\geq n}, \inf\set{t \geq 0: \abs{\Theta_t} \geq n}}.
\end{align*}
Then for the stopped process $(X_t^{\tau_n}, \Theta_t^{\tau_n})_{t \geq 0}$ it holds that
\begin{align*}
     \mathbb{E}\left[\sup_{t \in [0,T]}\norm{X_t^{\tau_n}}^2+\abs{\Theta_t^{\tau_n}}^2\right] 
     \leq 
     \mathbb{E}\left[\norm{X_0}^2+\abs{\Theta_0}^2\right] + 2n^2. 
\end{align*}
Via Cauchy-Schwarz and the Burkholder-Davis-Gundy inequality one can now derive the upper bound 
\begin{align*}
    \mathbb{E}\left[\sup_{t \in [0,T]}\norm{X_t^{\tau_n}}^2+\abs{\Theta_t^{\tau_n}}^2\right] 
    \leq 
    C \left(\mathbb{E}\left[\norm{X_0}^2+\abs{\Theta_0}^2\right] + K^2(T^2+ T)\right), 
\end{align*}
where the constant $C$ only depends on $T$ and $K$. The right hand side of this inequality does not depend on $n$, so we can conclude that the unstopped process $X$ also satisfies this inequality by letting $n$ tend to infinity. 
\end{proof}

\begin{proof}[Proof of Theorem \ref{theorem:existence-uniqueness}]
As already announced, the proof of the existence is based on the fixed point argument. Let us define a map $\Phi: \calP^2 \to \calP^2$ by mapping $\mu \in \calP^2$ to the law of the solution of the coupled system of SDEs
\begin{align*}
\begin{cases}
     dX^{\mu}_t 
     &= 
     \left(\int_{\mathcal{C}}U(X^{\mu}_t, y, \Theta^{\mu}_t, \theta) \mu(dy d\theta)\right)dt 
     +
    \left(\int_{\mathcal{C}}\sigma_{sp}(X_t^\mu, y, \Theta_t^\mu, \theta)\mu(dyd\theta)\right) dB^{sp}_t, 
    \\\
    d\Theta^{\mu}_t 
    &= 
    \left(\int_{\mathcal{C}}V(X^{\mu}_t, y, \Theta^{\mu}_t, \theta) \mu(dy d\theta)\right)dt 
    +
    \left(\int_{\mathcal{C}}\sigma_{op}(X_t^\mu, y, \Theta_t^\mu, \theta)\mu(dyd\theta)\right) dB^{op}_t, 
    \\\
    X^{\mu}_{t=0} &= X_0,  
    \\\
    \Theta^{\mu}_{t=0} &= \Theta_0. 
\end{cases}
\end{align*}
The existence and uniqueness of a solution to this system follows from the fact that for fixed $\mu$ the drift and diffusion coefficient are Lipschitz continuous and bounded. Moreover,  from standard a priori estimates for classical SDEs  it follows that $\Phi(\mu) \in \calP^2$.

It remains to show, that $\Phi$ is indeed a contraction on $\calP^2$. To see this, fix two probability measures $\mu^1, \mu^2 \in \calP^2$ and let $(X^1, \Theta^1),(X^2,\Theta^2)$ be the corresponding solutions, i.e. for $i=1,2$ we have 
\begin{align*}
\begin{cases}
    dX^i_t 
    &= 
    \left(\int_{\mathcal{C}}U(X^i_t, y, \Theta^i_t, \theta) \mu^i(dy d\theta)\right)dt
    + 
    \left(\int_{\mathcal{C}}\sigma_{sp}(X_t^i, y, \Theta_t^i, \theta)\mu^i(dyd\theta)\right) dB^{sp}_t, 
    \\\
    d\Theta^i_t 
    &= 
    \left(\int_{\mathcal{C}}V(X^i_t, y, \Theta^i_t, \theta) \mu^i(dy d\theta)\right)dt 
    + 
    \left(\int_{\mathcal{C}}\sigma_{op}(X_t^i, y, \Theta_t^i, \theta)\mu^i(dyd\theta)\right) dB^{op}_t. 
\end{cases}
\end{align*}
To estimate $D^2_T(\Phi(\mu^1), \Phi(\mu^2))$, we will use that the law of $(X^1, \Theta^1, X^2,\Theta^2)$ is actually a coupling of the two probability measures $\Phi(\mu^1)$ and $\Phi(\mu^2)$. Therefore, by definition of the $2$-Wasserstein distance as the infimum over all couplings we have
\begin{equation*}
        D^2_T(\Phi(\mu^1), \Phi(\mu^2))
    \leq 
    \mathbb{E}
    \left[
        \sup_{s\leq T}\norm{X^1_s -X^2_s}^2 + \abs{\Theta^1_s - \Theta^2_s}^2
    \right].
\end{equation*}
We will now set up a Gronwall-type argument to obtain an upper bound of the right-hand side of this inequality in terms of the $2$-Wasserstein distance of $\mu^1$ and $\mu^2$. 
As a first step, we use Hölder's  and Young's inequality 
to see that for fixed $s\leq T$ we have 
\begin{align*}
    |&X_s^1 - X^2_s|^2 
    \leq 
    2s \int_0^s \abs{\int_{\mathcal{C}}\!\!U(X^1_r, y_r, \Theta^1_r, \theta_r) \mu^1(dy d\theta) 
    - 
    \int_{\mathcal{C}}\!\!U(X^2_r, y_r, \Theta^2_r, \theta_r) \mu^2(dy d\theta)}^2
    ds \\
    +
    &2 
    \abs{
        \int_0^s
            \left(
                \int_{\mathcal{C}}\sigma_{sp}(X_r^1, y_r, \Theta_r^1, \theta_r)\mu^1(dyd\theta)
                -
                \int_{\mathcal{C}}\sigma_{sp}(X_r^2, y_r, \Theta_r^2, \theta_r)\mu^2(dy d\theta)
            \right)
        dB^{sp}_r
    }^2
    \end{align*}

where $y_r, \theta_r$ are the projections onto the time $r$ coordinate of the $d$-dimensional spatial component and the $1$-dimensional opinion component respectively. We obtain the analogoue estimate for $ \abs{\Theta_s^1-\Theta_s^2}^2$.

Utilizing then   Doob's $L^2-$inequality to estimate the stochastic part and Lipschitz-type inequality (\ref{wasserstein-lipschitz-inequality}), enable us to apply  Gronwall's lemma which implies that  for all $t \in [0,T]$ we have

\begin{align} \label{Gronwall_estimate}
    \mathbb{E}
    \left[
        \sup_{s\leq t}\norm{X^1_s - X^2_s}^2+ \abs{\Theta^1_s - \Theta^2_s}^2
    \right]
    \leq 
    C \int_0^t\calW^2_{2,\R^{d+1}}(\mu^1_r, \mu^2_r)dr,
\end{align}
where the constant $C$ is given explicitly by 
\begin{align*}
    C:= 2(8 + 2T)L^2 \exp(2(8+2T)L^2) 
\end{align*}
and $\calW_{2,\R^{d+1}}$ is the usual $2$-Wasserstein distance on $\calP^2(\R^{d+1})$. 
Next observe that for two measures $\mu, \nu \in \calP^2((C([0,T]; \R^{d+1}))$ and fixed $r \in [0,T]$ we have 
\begin{align}\label{projection-inequality}
    \calW^2_{2,\R^{d+1}}(\mu_r, \nu_r) 
    \leq 
    D_r^2(\mu, \nu).
\end{align}

Applying $(\ref{projection-inequality})$ to \eqref{Gronwall_estimate} gives us 
\begin{align*}
    \mathbb{E}
    \left[
        \sup_{s\leq t}\norm{X^1_s - X^2_s}^2+ \abs{\Theta^1_s - \Theta^2_s}
    \right]
    \leq 
    C \int_0^t D^2_r(\mu^1, \mu^2)dr.
\end{align*}
Furthermore, since the law of $(X^1, \Theta^1,X^2, \Theta^2)$ is a coupling of $\Phi(\mu^1), \Phi(\mu^2)$, the definition of the  $2$-Wasserstein distance \eqref{2Wass_def} implies 
\begin{align*}
    D^2_T(\Phi(\mu^1), \Phi(\mu^2)) 
    \leq 
    \mathbb{E}
    \left[
        \sup_{s\leq T}\norm{X^1_s - X^2_s}^2+ \abs{\Theta^1_s - \Theta^2_s}
    \right]
    \leq 
    C \int_0^T D^2_r(\mu^1_r, \mu^2_r)dr.
\end{align*}
After establishing this inequality, the well-posedness follows from the  usual application of Banach's fixed point theorem. 
By Lemma \ref{lemma:a-priori-estimate} this shows that the solution is unique in law. Now if we let $\mu$ be the unique law of the solution, and let it be the drift and diffusion coefficient of (\ref{lim_system}), then we have a standard SDE with bounded and Lipschitz continuous drift and diffusion coefficients. For such systems it is well known, that pathwise uniqueness holds. Therefore we can conclude that the solution to (\ref{lim_system}) is not only unique in law but also pathwise unique.  
\end{proof}

\subsection*{Appendix A.2.}

\begin{proof}[Proof of Theorem \ref{theorem:convergence-microscopic-model}]
For fixed $t\in [0,T]$ we have by Hölder's inequality
\begin{align*}
    &\norm{X^{i,N}_t - \overline{X}^i_t}^2 
    \\\
    \leq 
    &2t\int_0^t 
        \norm{
            \int_{\mathcal{C}}U(X^{i,N}_s, y, \Theta^{i,N}, \theta)\nu_N(dy d\theta) 
            -
            \int_{\mathcal{C}}U(\overline{X}^i_s, y, \overline{\Theta^{i,N}}, \theta)\mu(dy d\theta) 
        }^2
    ds
    \\\
    + 
    &2\norm{
        \int_0^t
            \left(
            \int_{\mathcal{C}}\sigma_{sp}(X^{i,N}_s, y, \Theta^{i,N}, \theta)\nu_N(dy d\theta) 
            -
            \int_{\mathcal{C}}\sigma_{sp}(\overline{X}^i_s, y, \overline{\Theta^{i,N}}, \theta)\mu(dy d\theta) 
            \right)
        dB^{sp}_s
    }^2.
\end{align*}
We obtain the analogue estimate for $\abs{\Theta^{i,N}_t -\overline{\Theta}_t^i}^2$.

Similarly as before, taking the supremum of previous estimates and integrating with respect to $\mathbb{P}$ and  Doob's $L^2$-inequality  and the Lipschitz type inequality (\ref{wasserstein-lipschitz-inequality}) we obtain
\begin{align*}
    &\mathbb{E}\left[
    \sup_{0\leq s \leq t}\norm{X^{i,N}_s - \overline{X}^i_s}^2 + \abs{\Theta^{i,N}_s - \overline{\Theta}^i_s}^2
    \right]
   \\\
    \leq 
    &2(2t + 8)L^2 
    \mathbb{E}\left[
        \int_0^t
            \norm{X^{i,N}_s - \overline{X}^i_s}^2
            +
            \abs{\Theta^{i,N}_s - \overline{\Theta}^i_s}^2
            +
            \calW_{2,\R^{d+1}}^2(\mu^N_s, \mu_s)
        ds.
    \right],
\end{align*}
where $\calW_{2,\R^{d+1}}(\cdot, \cdot)$ is the usual $2$-Wasserstein metric on the set of probability measures on $\R^{d+1}$ with finite second moment. 
The application of Gronwall's inequality and the analogue arguments as in the proof of Theorem 1, yield the estimate
\begin{align*}
    \mathbb{E}\left[
    \sup_{0\leq s \leq t}\norm{X^{i,N}_s - \overline{X}^i_s}^2 + \abs{\Theta^{i,N}_s - \overline{\Theta}^i_s}^2
    \right]
    \leq
    C\mathbb{E}\left[
        \int_0^t D_s^2(\mu, \mu^N)ds
    \right]. 
\end{align*}
By construction we know that $(\overline{X}^i, \overline{\Theta}^i)_{i \in \N}$ is an $i.i.d.$ family with $\text{Law}(\overline{X}^i, \overline{\Theta}^i) = \mu$ for all $i \in \N$. We define the empirical measure of the $(\overline{X}^i, \overline{\Theta}^i)$'s by 
\begin{align}
    \overline{\mu}^N := \frac{1}{N}\sum_{i=1}^N\delta_{(\overline{X}^i, \overline{\Theta}^i)}. 
\end{align}
Note that the random measure $\frac{1}{N}\sum_{i=1}^N\delta_{(X^{i,N}, \Theta^{i,N}, \overline{X}^i, \overline{\Theta}^i)}$ is a random coupling of the random measures $\mu^N$ and $\overline{\mu}^N$. 
Therefore the definition of the truncated $2$-Wasserstein distance gives us the pathwise inequality
\begin{align}
    D_t^2(\mu^N, \overline{\mu}^N) 
    \leq 
    \frac{1}{N}\sum_{i=1}^N 
            \sup_{0\leq s \leq t }\left(\norm{X^{i,N}_s - \overline{X}^i_s}^2
            +
            \abs{\Theta^{i,N}_s - \overline{\Theta}^i_s}^2
            \right).
\end{align}
Hence,  by our previous considerations we have
\begin{align}
    \mathbb{E}\left[D_t^2(\mu^N, \overline{\mu}^N)\right]
    \leq 
    C\mathbb{E}\left[
        \int_0^t D_s^2(\mu, \mu^N)ds
    \right].
\end{align}
Moreover, combining this with the triangle inequality for the truncated $2$-Wasserstein distance \eqref{2Wass_def}, we obtain 
\begin{align*}
    \mathbb{E}\left[D_t^2(\mu, \mu^N)\right]
    &\leq 
    2\mathbb{E}\left[D_t^2(\mu, \overline{\mu}^N)\right]
    +
    2C 
    \mathbb{E}\left[
        \int_0^t D_s^2(\mu, \mu^N)ds
    \right].
\end{align*}
Now we can apply Gronwall's inequality and take $t=T$ to get 

\begin{align*}
    \mathbb{E}\left[D^2_T(\mu, \mu^N)\right]
    \leq 
    2e^{2 C T}
    \mathbb{E}
    \left[
        D_T^2(\mu, \overline{\mu}^N)
    \right],
\end{align*}
where the right hand side tends to $0$ as $N\to \infty$ by the LLN for empirical measures of  i.i.d. random variables, cf. \cite[Theorem 11.4.1]{dudley_real_2002}. 
\newline 
Since this LLN only gives us the almost sure convergence of the random measures $\mu^N \to \mu$, 
it remains to show that the family $(D_T^2(\mu, \mu^N))_{N\in \N}$ is uniformly integrable.
 
This can be shown by observing that  
\begin{align}
    D_T^2(\mu, \overline{\mu}^N) 
    \leq 2 D^2_T(\mu, \delta_0) + 2 \calW^2_2(\delta_0, \overline{\mu}^N) 
    \leq 2 D^2_T(\mu, \delta_0) + 2 \frac{1}{N}\sum_{j=1}^N \sup_{s\leq T}\norm{(\overline{X}^j_s, \overline{\Theta}^j_s)}^2,
\end{align}
where $\delta_0$ is the Dirac zero measure on $C([0,T]; \R^{d+1})$.
The right hand side of this inequality is uniformly integrable, since the first term does not depend on $N$ 
and the family $(\frac{1}{N}\sum_{j=1}^N \sup_{s\leq T}\norm{(\overline{X}^j_s, \overline{\Theta}^j_s)}^2)_{N\in \N}$ is uniformly integrable as the empirical average of i.i.d. copies of an integrable random variable. Then the uniform integrability of the empirical averages follows directly from the definition. 

\end{proof}


\bibliography{Literature}
\bibliographystyle{unsrt}

\end{document}